\pdfoutput=1
\RequirePackage[l2tabu,orthodox]{nag}
\documentclass[a4paper,10pt]{amsart}
\date{\today}
\title[Characteristic forms]%
{Characteristic forms of complex Cartan geometries II}
\newcommand{\authorsname}{\texorpdfstring{Benjamin \scotsMc{}Kay}{Benjamin McKay}}
\author{\authorsname}
\address{School of Mathematical Sciences,  University College Cork, Cork, Ireland}
\email{b.mckay@ucc.ie}
\thanks{Thanks to Anca Musta\c{t}\u{a} and Andrei Musta\c{t}\u{a} for help with algebraic geometry, to Francesco Russo for his invitation to carry out this work at the University of Catania, to Tatsuo Suwa and Filippo Bracci for explaining \Cech--Dolbeault cohomology and to Robert Bryant for his insights into work of Cartan.}

\makeatletter
\DeclareRobustCommand{\scotsMc}{\scotsMcx{c}}
\DeclareRobustCommand{\scotsMC}{\scotsMcx{\textsc{c}}}
\DeclareRobustCommand{\scotsMcx}[1]{%
  M%
  \raisebox{\dimexpr\fontcharht\font`M-\height}{%
    \check@mathfonts\fontsize{\sf@size}{0}\selectfont
    \kern.3ex\underline{\kern-.3ex #1\kern-.3ex}\kern.3ex
  }%
}
\expandafter\def\expandafter\@uclclist\expandafter{%
  \@uclclist\scotsMc\scotsMC
}
\makeatother

\usepackage{lmodern}
\usepackage[T2A,T1]{fontenc}
\usepackage[utf8]{inputenc}
\usepackage[tt={variable=false}]{cfr-lm}
\usepackage{etex}
\usepackage{xparse}
\usepackage{mathtext}
\usepackage{mathrsfs}
\usepackage[english]{babel}
\usepackage{varioref}
\usepackage[pdftex,pagebackref]{hyperref}
\hypersetup{
colorlinks   = true,  %Colours links instead of ugly boxes
urlcolor     = black, %Colour for external hyperlinks
linkcolor    = black, %Colour of internal links
citecolor    = black  %Colour of citations
}
\usepackage{cleveref}
\usepackage[kerning=true,tracking=true]{microtype}
\usepackage{xspace}
\usepackage{amsfonts}
\usepackage{verbatim}
\usepackage{amsthm}
\usepackage{amssymb}
\usepackage{mathtools}
\usepackage{mathabx}
\usepackage{braket}
\usepackage{cool}
\usepackage{xstring}
\usepackage{array}
\usepackage{ragged2e}
\usepackage{longtable}
\usepackage{multirow}
\usepackage{booktabs}
\usepackage{bm}
\usepackage{tikz-cd}
\usepackage{varwidth}
\usepackage{tensor}
\usepackage[super]{nth}
\usepackage{scalerel}
\usepackage{rank-2-roots}
\usepackage{dynkin-diagrams}
\newtheorem{theorem}{Theorem}[section]
\newtheorem{corollary}[theorem]{Corollary}
\newtheorem{lemma}[theorem]{Lemma}
\newtheorem{proposition}[theorem]{Proposition}
\theoremstyle{remark}

\newtheorem{example}[theorem]{Example}
\newtheorem{remark}[theorem]{Remark}

\newcounter{remarkCounter}
\NewDocumentCommand\map{omm}%
{%
\IfValueTF{#1}%
{{#2}\xrightarrow{#1}{#3}}%
{{#2}\to{#3}}%
}%

\NewDocumentCommand\mapto{omm}%
{%
\IfValueTF{#1}%
{{#2}\xmapsto{#1}{#3}}%
{{#2}\mapsto{#3}}%
}%

\NewDocumentCommand\pr{m}{\ensuremath{\left(#1\right)}}
\NewDocumentCommand\of{m}{\ensuremath{\!\pr{#1}}}
\RenewDocumentCommand\C{o}%
{%
	\ensuremath%
	{\IfValueTF{#1}{\mathbb{C}^{#1}}{\mathbb{C}}}%
}%
\NewDocumentCommand\OO{D(){0}O{1}}%
{%%
\ensuremath{%%%
\mathcal{O}%
\ifnum\pdfstrcmp{#1}{0}=0{}\else{(#1)}\fi
\ifnum\pdfstrcmp{#2}{1}=0{}\else{
	\ifnum\pdfstrcmp{#2}{0}=0{}\else{^{\oplus #2}}\fi
}\fi
}%%%
}%%
\ifdefined\Proj
\RenewDocumentCommand\Proj{m}{\ensuremath{\mathbb{P}^{#1}}}
\else
\NewDocumentCommand\Proj{m}{\ensuremath{\mathbb{P}^{#1}}}
\fi
\NewDocumentCommand\Sym{mm}{\ensuremath{\operatorname{Sym}^{#1}\!\of{#2}}}
\NewDocumentCommand\GL{m}{\ensuremath{\operatorname{GL}_{#1}}}
\NewDocumentCommand\Lie{mo}{\ensuremath{\mathfrak{\MakeLowercase{#1}}\IfValueT{#2}{_{#2}}}}
\NewDocumentCommand\LieGL{m}{\Lie{GL}[#1]}
\NewDocumentCommand\SL{m}{\ensuremath{\operatorname{SL}_{#1}}}
\NewDocumentCommand\LieSL{m}{\Lie{SL}[#1]}
\NewDocumentCommand\LieSO{m}{\Lie{SO}[#1]}
\NewDocumentCommand\PSL{m}{\ensuremath{\mathbb{P}\!\operatorname{SL}_{#1}}}
\NewDocumentCommand\PO{m}{\ensuremath{\mathbb{P}\operatorname{O}_{#1}}}

\NewDocumentCommand\Gr{mm}{\ensuremath{\operatorname{Gr}_{#1}{#2}}}
\NewDocumentCommand\nForms{omm}%%
{%
\IfValueTF{#1}%%
{\ensuremath{\vb{#1}^{#2}_{#3}}}
{\ensuremath{\Omega^{#2}_{#3}}}
}%
\NewDocumentCommand\Lm{smm}{\ensuremath{\Lambda^{#2}\IfBooleanTF{#1}{\!\left(#3\right)}{#3}}}
\NewDocumentCommand\cohomology{omm}{\ensuremath{H_{\IfValueT{#1}{\IfStrEq{#1}{db}{\bar\partial}{#1}}}^{#2}\of{#3}}}
\DeclareMathOperator{\Ad}{Ad}
\DeclareMathOperator{\ad}{ad}
\NewDocumentCommand\LieDer{}{\ensuremath{\mathcal L}}
\NewDocumentCommand\hook{}{\ensuremath{\mathbin{ \hbox{\vrule height1.4pt
        width4pt depth-1pt \vrule height4pt width0.4pt depth-1pt}}}}
\NewDocumentCommand\Orth{m}{\ensuremath{\operatorname{O}_{#1}}}
\NewDocumentCommand\defeq{}{\coloneqq}
\NewDocumentCommand\drop{mm}{\ensuremath{#2/#1}}
\NewDocumentCommand\smooth{}{\ensuremath{C^{\infty}}\xspace}
\NewDocumentCommand\MakeLie{m}{\expandafter\def\csname Lie#1\endcsname{\Lie{#1}}}
\def\lst{A,B,G,H,K,L,P,Q,R,S,T,U,Z}
\makeatletter
\@for\i:=\lst\do{\expandafter\MakeLie \i}
\makeatother
\NewDocumentCommand\lb{smm}%
{%
\IfBooleanTF{#1}%
{%%
\ensuremath{\left[{#2},{#3}\right]}
}%%
{%%
\ensuremath{[{#2}{#3}]}
}%%
}%
\NewDocumentCommand\prol{omo}{#2^{(\IfValueTF{#1}{#1}{1})\IfValueT{#3}{#3}}}
\NewDocumentCommand\transpose{m}{\tensor[^\top]{\!#1}{}}
\NewDocumentCommand\At{o}{\operatorname{At}_{\IfValueT{#1}{#1}}}
\newcommand\varul[2][3]{\mkern#1mu\underline{\mkern-#1mu#2\mkern-#1mu}\mkern#1mu}
\NewDocumentCommand\qH{}{\varul{H}}
\NewDocumentCommand\LieqH{}{\varul[1]{\LieH}}
\NewDocumentCommand\qom{}{\varul[2]\omega}
\NewDocumentCommand\qv{}{\varul[2]{v}}

\NewDocumentCommand\qe{}{\varul[2]{e}}
\NewDocumentCommand\qE{}{\varul{E}}

\NewDocumentCommand\CdB{}{\ensuremath{\check\partial}}
\NewDocumentCommand\CSl{}{\ensuremath{\check\nabla}}
\NewDocumentCommand{\dS}{}{\nabla}
\NewDocumentCommand{\Dlb}{sm}%
{%
\IfBooleanTF{#1}%
{%%
\Dlb{\left(#2\right)}%
}%%
{%%
#2^{\bar\partial}%
}%%
}%
\NewDocumentCommand{\Cch}{sm}%
{%
\IfBooleanTF{#1}%
{%%
\left(#2\right)^{\vee}%
}%%
{%%
\check{#2}%
}%%
}%
\NewDocumentCommand\Slovak{}{Slov\'ak}
\NewDocumentCommand\Cech{}{\u{C}ech}
\NewDocumentCommand\deRham{}{de~Rham}
\NewDocumentCommand\inci{}{\hat\imath}
\NewDocumentCommand\cpx{om}{\ensuremath{#2_{\bullet\IfValueT{#1}{-#1}}}}
\NewDocumentCommand\gLm{omom}%
{%
\IfValueTF{#1}%
{%%
{\varul[4]{#1}}%
}%%
{%%
\varul{\Lambda}%
}%%
^{#2,\IfValueTF{#3}{#3,#4}{#4}}%
}%

\NewDocumentCommand\gForms{omom}%
{%
\IfValueTF{#1}%
{%%
{\varul[4]{\vb{#1}}}%
}%%
{%%
\varul{\Omega}%
}%%
^{#2,\IfValueTF{#3}{#3,#4}{#4}}%
}%

\NewDocumentCommand\ch{mo}{\operatorname{ch}_{#1}\IfValueT{#2}{\!\left(#2\right)}}
\NewDocumentCommand\chp{mo}{\operatorname{ch}'_{#1}\IfValueT{#2}{\!\left(#2\right)}}
\NewDocumentCommand\om{m}{ω{#1}}
\NewDocumentCommand\w{}{\wedge}
\NewDocumentCommand\vp{m}{ϖ{#1}}
\newcommand{\amal}[3]{\ensuremath{{#1} \mathbin{\times^{#2}} \! #3}}
\NewDocumentCommand\vb{m}{\ensuremath{\bm{#1}}}
\NewDocumentCommand\Conn{m}{\ensuremath{\mathscr{A}_{#1}}}
\makeatletter
\def\l@subsection{\@tocline{2}{0pt}{2.5pc}{5pc}{}}
\def\l@section{\@tocline{1}{0pt}{2.5pc}{5pc}{}}
\def\@tocline#1#2#3#4#5#6#7{\relax
  \ifnum #1>\c@tocdepth % then omit
  \else
    \par \addpenalty\@secpenalty\addvspace{#2}%
    \begingroup \hyphenpenalty\@M
    \@ifempty{#4}{%
      \@tempdima\csname r@tocindent\number#1\endcsname\relax
    }{%
      \@tempdima#4\relax
    }%
    \parindent\z@\relax
\leftskip#3\relax \advance\leftskip\@tempdima\relax
    #5\leavevmode\hskip-\@tempdima{#6}\nobreak\relax
    ,~#7\par
    \nobreak
    \endgroup
  \fi}
\makeatother
\NewDocumentCommand\CS{mo}{T_{#1\IfValueT{#2}{,#2}}}
\NewDocumentCommand\proj{om}{\operatorname{proj}_{\IfValueT{#1}{#1}}\!\left(#2\right)}
\NewDocumentCommand\exactSeq{smO{}mO{}m}%
{%
\begin{tikzcd}[cramped, sep=small,ampersand replacement=\&]
\IfBooleanTF{#1}{1}{0} \rar \& #2 \rar{#3} \& #4 \rar{#5} \& #6 \rar \& \IfBooleanTF{#1}{1}{0}
\end{tikzcd}%
}%
\NewDocumentCommand\bundle{smO{}mO{}m}%
{%
\IfBooleanTF{#1}%
{%
\begin{tikzcd}[cramped, sep=small,ampersand replacement=\&]
#2 \rar{#3} \& #4 \rar{#5} \& #6
\end{tikzcd}%
}%
{%
\begin{tikzcd}[cramped, sep=small,ampersand replacement=\&]
#2 \rar{#3} \& #4 \dar{#5} \\ \& #6
\end{tikzcd}%
}%
}%
\ifdefined\DeclareUnicodeCharacter
\DeclareUnicodeCharacter{028C}{\wedge}
  \DeclareUnicodeCharacter{0393}{\Gamma}         % Γ
  \DeclareUnicodeCharacter{0394}{\Delta}         % Δ
  \DeclareUnicodeCharacter{0398}{\Theta}         % Θ
  \DeclareUnicodeCharacter{039B}{\Lambda}        % Λ
  \DeclareUnicodeCharacter{039E}{\Xi}            % Ξ
  \DeclareUnicodeCharacter{03A0}{\Pi}            % Π
  \DeclareUnicodeCharacter{03A3}{\Sigma}         % Σ
  \DeclareUnicodeCharacter{03A5}{\Upsilon}       % Υ
  \DeclareUnicodeCharacter{03A6}{\Phi}           % Φ
  \DeclareUnicodeCharacter{03A8}{\Psi}           % Ψ
  \DeclareUnicodeCharacter{03A9}{\Omega}         % Ω
  \DeclareUnicodeCharacter{03B1}{\alpha}         % α
  \DeclareUnicodeCharacter{03B2}{\beta}          % β
  \DeclareUnicodeCharacter{03B3}{\gamma}         % γ
  \DeclareUnicodeCharacter{03B4}{\delta}         % δ
  \DeclareUnicodeCharacter{03B5}{\varepsilon}    % ε textepsilon/varepsilon
  \DeclareUnicodeCharacter{03B6}{\zeta}          % ζ
  \DeclareUnicodeCharacter{03B7}{\eta}           % η
  \DeclareUnicodeCharacter{03B8}{\theta}         % θ
  \DeclareUnicodeCharacter{03B9}{\iota}          % ι
  \DeclareUnicodeCharacter{03BA}{\kappa}         % κ
  \DeclareUnicodeCharacter{03BB}{\lambda}        % λ
  \DeclareUnicodeCharacter{03BC}{\mu}            % μ
  \DeclareUnicodeCharacter{03BD}{\nu}            % ν
  \DeclareUnicodeCharacter{03BE}{\xi}            % ξ
  \DeclareUnicodeCharacter{03C0}{\pi}            % π
  \DeclareUnicodeCharacter{03C1}{\rho}           % ρ
  \DeclareUnicodeCharacter{03C2}{\varsigma}      % ς
  \DeclareUnicodeCharacter{03C3}{\sigma}         % σ
  \DeclareUnicodeCharacter{03C4}{\tau}           % τ
  \DeclareUnicodeCharacter{03C5}{\upsilon}       % υ
  \DeclareUnicodeCharacter{03C6}{\varphi}        % φ textphi/varphi
  \DeclareUnicodeCharacter{03C7}{\chi_}           % χ
  \DeclareUnicodeCharacter{03C8}{\psi}           % ψ
  \DeclareUnicodeCharacter{03C9}{\omega_}         % ω
  \DeclareUnicodeCharacter{03D1}{\thetasymbol}   % ϑ
  \DeclareUnicodeCharacter{03D5}{\phisymbol}     % ϕ $\phi$
  \DeclareUnicodeCharacter{03D6}{\varpi_}      % ϖ
  \DeclareUnicodeCharacter{03DD}{\digamma}       % ϝ
  \DeclareUnicodeCharacter{03F1}{\rhosymbol}     % ϱ
  \DeclareUnicodeCharacter{03F5}{\epsilonsymbol} % ϵ $\epsilon$
\fi
\begin{document}
\begin{abstract}
Characteristic class relations in Dolbeault cohomology follow from the existence of a holomorphic Cartan geometry (for example, a holomorphic conformal structure or a holomorphic projective connection).
These relations can be calculated directly from the representation theory of the structure group, without selecting any metric or connection or having any knowledge of the Dolbeault cohomology groups of the manifold.
This paper improves on its predecessor \cite{McKay:2011} by allowing noncompact and non-K\"ahler manifolds and by deriving invariants in cohomology of vector bundles, not just in scalar Dolbeault cohomology, and computing relations involving Chern--Simons invariants in Dolbeault cohomology.
For the geometric structures previously considered in its predecessor, this paper gives stronger results and simplifies the computations.
It gives the first results on Chern--Simons invariants of Cartan geometries.
\end{abstract}
\maketitle
\begin{center}
\tableofcontents
\end{center}
\section{Introduction}
Geometers have developed an extensive knowledge of holomorphic geometric structures on smooth projective varieties, and on many other complex manifolds.
These geometric structures are rare, while meromorphic geometric structures are far too common.
So is it natural to seek important special examples of meromorphic geometric structures.
We prepare for the arrival of such structures by considering how we might compute residue integrals of invariants, e.g. Chern--Simons invariants, around their singularities.

We prove that any invariant polynomial relation between Atiyah classes of homogeneous holomorphic vector bundles on the model of a holomorphic Cartan geometry is also satisfied on that holomorphic Cartan geometry in Dolbeault cohomology.
Our work is similar to the construction of \cite{Eastwood/Gindikin/Wong:1996}, building a holomorphic cohomology theory directly in the structure equations of these geometric structures.
We prove the only known results about Chern--Simons invariants of Cartan geometries: computations of previously unsuspected polynomial relations, which we find are explicitly computable, sometimes by hand, in terms of Cartan's structure equations.
\section{Holomorphic Cartan geometries}
\subsection{Notation}
Denote the Lie algebra of a Lie group \(H\) as \(\LieH\), and similarly denote the Lie algebra of any Lie group by the corresponding fraktur font expression.
All Lie groups are complex analytic and finite dimensional.
All \(H\)-modules are finite dimensional and holomorphic.
Denote the left invariant Maurer--Cartan \(1\)-form on \(H\) as \(h^{-1}dh\).
Take a holomorphic right principal bundle \bundle*{H}{E}{M}.
For each vector \(v \in \LieH\), denote also by \(v\) the associated vector field on \(E\): for any \(x \in E\),
\[
v(x) = \left.\frac{d}{dt}\right|_{t=0} x \, e^{tv}.
\]
If we wish to be more precise, we denote the vector field \(v\) on \(E\) as \(v_E\).
For any \(H\)-action on a manifold \(X\), denote by \(E \times^H X\) the quotient of \(E \times X\) by the diagonal right \(H\)-action \((e,x)h=(eh,h^{-1}x)\).
If \(V\) is a complex analytic \(H\)-module, denote the associated vector bundle by \(\vb{V} \defeq E \times^H V\).
\subsection{Cartan geometries}
A \emph{complex homogeneous space} \((X,G)\) is a complex manifold \(X\) and a complex Lie group \(G\) with a holomorphic transitive action on \(X\).
A \emph{morphism} of complex homogeneous spaces \(\map[(\varphi,\Phi)]{(X,G)}{(X',G')}\) is a holomorphic map \(\map[\varphi]{X}{X'}\) equivariant for a holomorphic Lie group morphism \(\map[\Phi]{G}{G'}\).
Pick a point \(x_0\in X\) and let \(H\defeq G^{x_0}\) be the stabilizer.
Denote by \(\LieG\) be the Lie algebra of \(G\).
A \emph{holomorphic Cartan geometry}, modelled on \((X,G)\), on a complex manifold \(M\) is a holomorphic principal \(H\)-bundle \(H\to E\to M\) and a holomorphic connection \(\omega\in\nForms{1}{E_G}\otimes^G\LieG\) on the associated principal \(G\)-bundle \(E_G\defeq\amal{E}{H}{G}\) so that the inclusion \(\map{E}{E_G}\) is complementary to the \emph{horizontal space} \((\omega=0)\subset TE_G\) \cite{Cap.Slovak.2009a,mckay2023introduction}.
We also denote by \(\omega\) the pullback of \(\omega\) to \(E\).
\subsection{Soldering}
Let \(\LieG_-\defeq\LieG/\LieH\).
The \emph{soldering form} is the \(1\)-form \(ω-\defeq\omega+\LieH\in \amal{\nForms{1}{E}}{H}{(\LieG/\LieH)}\).
Let \(\map[\pi_M]{E}{M}\) be the obvious quotient map \(M=E/H\).
The soldering form yields a surjective vector bundle morphism
\[
\map{TM}{\amal{E}{H}{\LieG_-}}
\]
by taking each \(e \in E\) and vector \(v \in T_e E\), say with \(\pi_E(e)=m \in M\), to both \(\pi_E'(e)v \in T_m M\) and to \((e,w) \in E \times \LieG_-\) by \(w=v \hook \omega+\LieH\).
Let \(H_+\subseteq H\) be a closed complex subgroup acting trivially on \(\LieG_-\).
Let \(\qH\defeq H/H_+\) and \(\qE\defeq E/H_+\).
This morphism descends to
\[
TM\cong\amal{\qE}{\qH}{\LieG_-}.
\]
Motivated by this isomorphism, our goal is not to find characteristic classes of \(E\) but of \(\qE\).
\subsection{Curvature}
Since \(\omega\) is a connection on \(E_G\), its curvature
\[
\Omega\defeq d\omega+\frac{1}{2}\lb{\omega}{\omega}
\]
is semibasic, so pulls back to \(E\) to be a multiple of
\[
ω-^2\in\nForms{2}{E}\otimes^H\Lm{2}{\LieG_-},
\]
say \(\Omega=Kω-^2\).

\subsection{Langlands decompositions}
A \emph{Langlands decomposition} of a complex Lie group \(H\) is a semidirect product decomposition \(H = H_0 \ltimes H_+\) in closed complex subgroups, where \(H_+\) is a connected and simply connected solvable complex Lie group and \(H_0\) is a reductive complex linear algebraic group.
For example:
\begin{enumerate}
\item
This definition generalizes the usual Langlands decomposition of any parabolic subgroup of any complex semisimple Lie group \cite{Knapp:2002} p. 481.
\item
Every connected and simply connected complex Lie group \(H\) admits a Langlands decomposition in which \(H_0\) is a maximal semisimple subgroup and \(H_+\) is the solvradical
\cite{Varadarajan:1984} p. 244 theorem 3.18.13.
\item
Any connected complex Lie group \(H\) admits a faithful holomorphic representation just when it admits a Langlands decomposition in which \(H_0\) is a complex linearly reductive group and \(H_+\) is the nilradical
\cite{Hilgert.Neeb:2012} p. 595 theorem 16.2.7.
\item
Every complex linear algebraic group \(H\) (perhaps disconnected) admits a Langlands decomposition in which \(H_0 \subset H\) is a maximal reductive subgroup and \(H_+ \subset H\) is the unipotent radical \cite{Hochschild:1981} p. 117 theorem 4.3.
In all of our examples below, \(H\) will be complex linear algebraic.
\end{enumerate}
Every connected and simply connected solvable complex Lie group \(H_+\) is biholomorphic to complex affine space \cite{Hilgert.Neeb:2012} p. 543 theorem 14.3.8, and so is a contractible Stein manifold.

Take a complex homogeneous space \((X,G)\), a point \(x_0\in X\), define \(H\defeq G^{x_0}\) and denote the Lie algebra of \(H\subseteq G\) as \(\LieH\subseteq\LieG\).
A \emph{Langlands decomposition} of a complex homogeneous space \((X,G)\) with \(X=G/H\) is a Langlands decomposition \(H = G_0 \ltimes G_+\) together with a splitting \(\LieG=\LieG_-\oplus \LieG_0\oplus \LieG_+\) into \(G_0\)-modules.
We will usually insist also that \(G_+\) acts on \(\LieG_-\defeq\LieG/\LieH\) preserving a filtration and acting trivially on the associated graded, so that the tangent bundle of \(X\) is filtered, with associated graded an associated vector bundle of the principal right bundle \(\map{G_0=H/G_+}{G/G_+}\).

\subsection{Infinitesimal characteristic forms}
Take a complex homogeneous space with Langlands decomposition, in our notation as above.
Take a finite dimensional \(\LieG_0\)-module \(V\).
The \emph{Atiyah form} \(a=a_V\) is the element
\[
a\in\LieG_+^*\otimes\LieG_-^*\otimes\LieGL{V}.
\]
given by
\[
a(x,y)=-\rho_V \circ \proj[\LieG_0]{[xy]},
\]
for \(x\in\LieG_+, y\in\LieG_-\).
If \(V\) is not specified, we take \(V\defeq\LieG_0\) in the adjoint representation.
(Note that we do \emph{not} require that \(V\) be a \(G_0\)-module, so there might not be an associated vector bundle.)
The \emph{Chern forms} \(c_k\) are
\[
c_k \in \Sym{k}{\LieG_+ \otimes \LieG_-}^*
\]
given by
\[
\det\left(I+\frac{i}{2\pi}a\right)=1+c_1+c_2+\dots=c.
\]
Analogously define the Chern character forms and Todd forms.
More generally, if \(f\) is a \(\LieG_0\)-invariant complex symmetric multilinear form on \(\LieG_0\), say of degree \(k\), we associate to \(f\) the element, denoted by the same name,
\[
f \in \Sym{k}{\LieG_+ \otimes \LieG_-}^*
\]
given by
\[
f(a,\dots,a).
\]
The \emph{Chern--Simons form} of \(f\) is
\[
\CS{f}\of{u,v,w_+,w_-}\\
\defeq \sum_{j=0}^{k-1} a_j
f(
	u,
	\underbrace{v,\dots,v}_{j},
	\underbrace{a\of{w_+,w_-},\dots,a\of{w_+,w_-}}_{k-j-1}%
)
\]
where
\[
a_j\defeq\frac{(-1)^j(k-1)!}{%2^j
(k+j)!(k-1-j)!}
\]
and
\begin{align*}
u,v&\in\LieG_0, \\
w_+&\in\LieG_+,\\
w_-&\in\LieG_-.
\end{align*}
(N.B. the expression for \(a_j\) is not the same as in the paper of Chern and Simons \cite{Chern/Simons:1974}; their \(A_j\) is \(A_j=a_j/2^j\).)
The splitting principle: if \(0 \to U \to V \to W \to 0\) is an exact sequence of \(\LieG_0\)-modules, extend a basis of \(U\) into a basis of \(V\), so
\[
a_V=
\begin{pmatrix}
a_U & * \\
0 & a_W
\end{pmatrix}
\]
and compute the determinant: \(c(U)c(W)=c(V)\).
The \emph{tangent bundle Atiyah form} is
\[
\mapto[a_T]{x \in \LieG_+, y \in \LieG_-, z \in \LieG-}{\frac{a(x,y)z+a(x,z)y}{2}}.
\]

\section{Characteristic classes}
\subsection{The Atiyah bundle}\label{subsection:connection.bundle}
We review some well known material to establish notation and terminology, following the standard references \cite{Atiyah1957,Chern/Simons:1974}.
Take a holomorphic right principal bundle
\[
\bundle{H}{E}[\pi]{M}.
\]
Let \(\ad_E\defeq \amal{E}{H}{\LieH}\).
The \(H\)-invariant exact sequence
\[
\exactSeq{\ker \pi'}{TE}[\pi']{\pi^* TM}
\]
of vector bundles on \(E\) quotients by \(H\)-action to an exact sequence of vector bundles on \(M\):
\[
\exactSeq{\ad_E}{\At[E]}{TM}
\]
with middle term \(\At[E]=TE/H\) the \emph{Atiyah bundle}.
A holomorphic (\(C^{\infty}\)) splitting \(s\) of this exact sequence determines and is determined by a holomorphic (\(C^{\infty}\)) \((1,0)\)-connection \(\omega=ωs\) for the bundle \(\map{E}{M}\), which is the unique \((1,0)\)-form on \(E\) so that \(s(v)\hook\omega=0\) and \(w\hook\omega=w\) for \(w\in\LieH\), i.e. the splitting lifts each tangent vector to its horizontal lift \cite{Atiyah1957}.
Write the section as \(s=s_{\omega}\).

\subsection{The connection bundle}
The \emph{connection bundle} of \(E\) is the affine subbundle \(\Conn{E}\subset T^*M \otimes_M \At[E]\) consisting of complex linear maps which split the sequence over some point of \(M\); \(\map{\Conn{E}}{M}\) is a holomorphic bundle of affine spaces, modelled on the vector bundle \(T^*M \otimes_M \ad_E\).
So holomorphic \((or C^{\infty})\) \((1,0)\)-connections are precisely holomorphic \((or C^{\infty})\) sections of the connection bundle \(\Conn{E}\).
Differences of two connections lie in \(T^*M \otimes \ad_E\).
The \emph{Atiyah class} of \(H\to E\to M\) is the class in \(\cohomology{1}{M,T^*M \otimes \ad_E}\) given by differences of holomorphic connections on open subsets of \(M\).
Each element \(v\in\At[E,m]\) is an \(H\)-invariant section of \(\map{\left.TE\right|_{E_m}}{\left.\pi^*TM\right|_{E_m}}\).
Each fiber \(\Conn{E,m}\) is the set of all \(H\)-invariant sections \(\omega\) of
\[
\left.T^*E \otimes \LieH\right|_{E_m}
\]
so that \(v\hook\omega=v\) for \(v\in\LieH\) with \(H\)-invariance:
\[
ω{eh}=\Ad_h^{-1} r_h^{-1*}ωe
\]
for \(h\in H\).

Denote the bundle map as \(\map[\delta]{\Conn{E}}{M}\), with pullback
\[
\begin{tikzcd}
E\times_M\Conn{E} \rar{\Delta} \dar & E \dar \\
\Conn{E} \rar{\delta} & M.
\end{tikzcd}
\]
Each point \(x \in E\times_M\Conn{E}\) can be represented uniquely as \(x=(m,e,ω0)\) for some \(m \in M\), \(e \in E_m\), \(\map[ω0]{T_e E}{\LieH}\) so that \(w\hookω0=w\) for \(w \in \LieH\).
Conversely, every such triple \((m,e,\omega_0)\) represents a unique point \(x\in E\times_M\Conn{E}\).
The map \(\Delta\) is then expressed as \(\Delta(x)=\Delta(m,e,\omega_0)=e\).
The right action of \(H\) on \(E\times_M\Conn{E}\) is represented by
\[
(m,e,\omega_0)h=(m,eh,\omega_0')
\]
where
\[
\omega_0'=\Ad_h^{-1}(\omega_0\circ {r_h^{-1}}').
\]
There is a holomorphic connection \(\omega\) on \(E\times_M\Conn{E}\) defined for a tangent vector \(v \in T_x \left(E\times_M\Conn{E}\right)\) by \(v \hook \omega = (\Delta'(x)v)\hook ω0\) \cite{Biswas:2017}.
Given a holomorphic \((C^{\infty})\) \((1,0)\)-connection \(ω0\) on \(\map{E}{M}\), map \(\mapto[\Phi]{e \in E}{\Phi(e)\defeq(m,e,ω0)\in E\times_M\Conn{E}}\).
Compose with the bundle map \(\map{\pi}{E\times_M\Conn{E}}{\Conn{E}}\) to get a section of the connection bundle.
Pullback the bundle \(E\times_M\Conn{E}\) by the section to get a map \(\map[\varphi]{E}{\Conn{E}}\), so that \(\varphi^*E\times_M\Conn{E}=E\) has pullback connection \(\varphi^*\omega=ω0\).
The Dolbeault class \([(dω0)^{1,1}]\) is the Dolbeault representative of the Atiyah class.

The connection bundle of a vector bundle is the connection bundle of its associated principal bundle.

If \(V\) is a finite dimensional \(\LieH\)-module, it might not be an \(H\)-module, so there might be no associated vector bundle, but we can still define an Atiyah class.
Write our module's \(\LieH\)-action as \(\map[\rho]{\LieH}{\LieGL{V}}\).
The \(\LieH\)-module \(\rho(\LieH)\subset\LieGL{V}\) is an \(H\)-module, since \(\LieH\) is an \(H\)-module.
Each local choice of holomorphic connection \(\omega\) on \(E\) gives a differential form \(\rho\circ\omega\) valued in \(\rho(\LieH)\).
The \emph{Atiyah class} of \(H\to E\to M\) is the class \(a(M,\vb{V})\in\cohomology{1}{M,T^*M \otimes (\amal{E}{H}{\rho(\LieH)})}\) given by differences of such forms on overlaps of open subsets of \(M\), even though \(\vb{V}\) need not exist; we say that \(\vb{V}\) is the \emph{ghost} of a departed vector bundle.
For any \((1,0)\)-connection \(\omega\) on \(E\), the \((1,1)\)-form \((\rho\circ d\omega)^{1,1}\) is semibasic, valued in \(\amal{E}{H}{\rho(\LieH)}\), representing that Atiyah class in Dolbeault cohomology.

\subsection{Connection bundles of Cartan geometries}\label{subsection:ConnBundlesShapes}
Take a complex homogeneous space \((X,G)\) with Langlands decomposition.
Take a holomorphic Cartan geometry \(\map{E}{M}\) with model \((X,G)\).
Take a point \(m \in M\), a point \(e\in E_m\) and a tangent vector \(v \in T_m M\).
Let \(\qe\defeq eG_+=q(e)\).
Applying the Langlands decomposition to the Cartan connection \(\omega\), which is valued in \(\LieG\): write it as
\[
\omega=ω-+ω0+ω+,
\]
to define its projections to \(\LieG_-,\LieG_0,\LieG_+\).
Let \(\qE\defeq E/G_+\) and denote the projection map as \(\mapto[\pi_{\qE}]{e\in E}{\qe\in\qE}\).
The \(1\)-form \(ω0\), being semibasic, determines a covector \(\map[ωe]{T_{\qe} \qE}{\LieG_0}\) for each \(e \in E\), uniquely defined by \(\pi_{\qE}^*ωe=ω0\).
Map
\[
\mapto[\Phi]{e\in E}{(m,\qe,ωe)\in\qE\times_M\Conn{\qE}}
\]
which we quotient by \(G_0\)-action to get
\[
\begin{tikzcd}
E \arrow{d} \arrow{r}{\Phi} & \qE\times_M\Conn{\qE} \arrow{d} \arrow{r}{\Delta} & \qE \arrow{d} \\
E/G_0 \arrow{r}{\phi} & \Conn{\qE} \arrow{r}{\delta} & M
\end{tikzcd}
\]
applying the commutative diagram of \vref{subsection:connection.bundle} but to \(\qE\) instead of \(E\).
\begin{lemma}
\(\Phi^*\omega=ω0\).
\end{lemma}
\begin{proof}
For \(m \in M\), \(e \in E_m\), \(v\in T_e E\), let
\[
x\defeq(m,\qe,ωe)=\Phi(e)\in \qE\times_M\Conn{\qE}.
\]
so \(\Delta(x)=\Delta(m,\qe,ωe)=\qe\).
Compute
\begin{align*}
v \hook (\Phi^*\omega)_e
&=
\Phi'(e)v \hook ωx,
\\
&=
\Delta'(x)\Phi'(e)v \hook ωe,
\\
&=
(\Delta\circ\Phi)'(e)v \hook ωe,
\\
&=\pi_{\qE}'(e)v\hookωe,
\\
&=v\hook\pi_{\qE}^*ωe,
\\
&=v\hookω0.
\end{align*}
\end{proof}
Since \(\omega\) is a holomorphic connection, we find its curvature is \(\Omega\defeq d\omega+\frac{1}{2}\lb{\omega}{\omega}\), and pulls back to \(\Omega_0\defeq dω0+\frac{1}{2}\lb{ω0}{ω0}\), even if \(ω0\) is not a connection.
The Bianchi identity \(d\Omega=\lb{\Omega}{\omega}\) ensures the analogous identity \(d\Omega_0=\lb{\Omega_0}{ω0}\) even if \(ω0\) is not a connection and \(\Omega_0\) is not the curvature of a connection.

Similarly, for any \(G_0\)-invariant complex polynomial function \(\map[f]{\LieG_0}{\C{}}\), thought of as a symmetric multilinear form, the expression
\[
f_E\defeq f(\Omega_0,\dots,\Omega_0)
\]
is the pullback of the Chern form \(f_{\qE\times_M\Conn{\qE}}\) for the connection \(\omega\) on the bundle \(\map{\qE\times_M\Conn{\qE}}{\Conn{\qE}}\).
In particular, \(f_E\) is a closed holomorphic differential form.
Similarly define
\[
\CS{f}[E]\defeq \CS{f}(ω0,\lb{ω0}{ω0},ω+,ω-),
\]
which is the pullback of the Chern--Simons form \(\CS{f}[\qE\times_M\Conn{\qE}]\), hence \(d\CS{f}[E]=f_E\).

Note that \(\LieG_0\oplus\LieG_+=\LieH\subseteq\LieG\) is a Lie subalgebra, so
\begin{align*}
0&=\lb{\LieG_0}{\LieG_0}_-=0,\\
0&=\lb{\LieG_0}{\LieG_+}_-=0,\\
0&=\lb{\LieG_+}{\LieG_+}_-=0.
\end{align*}
Also \(\LieG_+\subseteq\LieH\) is an ideal so
\[
0=\lb{\LieG_+}{\LieG_+}_0.
\]
Finally, \(\LieG_-,\LieG_0,\LieG_+\) are \(G_0\)-modules, so
\begin{align*}
0&=\lb{\LieG_0}{\LieG_-}_0=0,\\
0&=\lb{\LieG_0}{\LieG_-}_+=0,\\
0&=\lb{\LieG_0}{\LieG_0}_+=0,\\
0&=\lb{\LieG_0}{\LieG_+}_0=0.
\end{align*}

Split by Langlands decomposition our expression for the curvature of any Cartan connection,
\[
\Omega=d\omega+\frac{1}{2}\lb{\omega}{\omega}=Kω-^2:
\]
as,
\begin{align*}
\Omega_-&=
dω-
+\frac{1}{2}\lb{ω-}{ω0}_-
+\frac{1}{2}\lb{ω0}{ω-}_-
+\frac{1}{2}\lb{ω+}{ω-}_-,
\\
&=K_-ω-^2-\frac{1}{2}\lb{ω-}{ω-}_-,\\
\\
\Omega_0
&=
dω0
+\frac{1}{2}\lb{ω0}{ω0}
+\frac{1}{2}\lb{ω+}{ω-}_0
+\frac{1}{2}\lb{ω-}{ω+}_0,
\\
&=
K_0ω-^2-\frac{1}{2}\lb{ω-}{ω-}_0,\\
\Omega_+
&=
dω+
+\frac{1}{2}\lb{ω-}{ω+}
+\frac{1}{2}\lb{ω+}{ω-}
+\frac{1}{2}\lb{ω0}{ω+}
+\frac{1}{2}\lb{ω+}{ω0}
+\frac{1}{2}\lb{ω+}{ω+},
\\
&=
K_+ω-^2-\frac{1}{2}\lb{ω-}{ω-}_+.
\end{align*}
Under action of any \(h_+\in G_+\),
\begin{align*}
r_{h_+}^*
\begin{pmatrix}
ω-\\
ω0\\
ω+
\end{pmatrix}
&=
r_{h_+}^*\omega,
\\
&=
\Ad_{h_+}^{-1}\omega,
\\
&=
\begin{pmatrix}
\pi_-h_+^{-1}&0&0\\
-\pi_0h_+&I&0\\
(I-\Ad_{h_+}^{-1})\pi_0h_+-\Ad_{h_+}^{-1}\pi_+h_+&\Ad_{h_+}^{-1}-I&\Ad_{h_+}^{-1}
\end{pmatrix}
\begin{pmatrix}
ω-\\
ω0\\
ω+
\end{pmatrix}.
\end{align*}
\subsection{Smooth reduction of structure group}
Take a complex homogeneous space \((X,G)\) with Langlands decomposition.
Take a holomorphic Cartan geometry \(\map{E}{M}\) with that model.
Split the Cartan connection \(\omega=ω-+ω0+ω+\).
Then \(ω- \in \nForms{1}{E} \otimes^H\LieG_-\) is the \emph{soldering form} and \(ω0 \in \nForms{1}{E} \otimes^{G_0}\LieG_0\) is the \emph{pseudoconnection form}.

Since \(H/G_0\) is contractible, \(\map{E/G_0}{M}\) admits a \smooth section \(\map[s]{M}{E/G_0}\) i.e. a \smooth \(G_0\)-reduction of structure group.
The \(1\)-form \(ω0\) on \(E\) pulls back to a \(1\)-form \(ω0\) on \(s^*E\).
Let \(\qE\defeq E/G_+\), a holomorphic principal right \(\qH\)-bundle \(\bundle*{\qH}{\qE}{M}\).
The \(1\)-form \(ω0\) extends from \(s^*E\) to a unique \(1\)-form on
\(\qE\cong s^*E\) which we also denote \(ω0\), and which satisfies \(v \hook ω0=v\) for \(v \in \LieG_0\).

\begin{lemma}
The \(1\)-form \(ω0\) on \(\qE\) associated to any \smooth (or holomorphic) \(G_0\)-reduction is a \smooth (holomorphic) \((1,0)\)-connection \(1\)-form.
\end{lemma}
\begin{proof}
Pick a point \((m,e)\in s^*E\), i.e. with \(m \in M\) and \(e \in E_m\) and \(eG_0=s(m)\).
So \(ω0\) at the corresponding point of \(\qE\) is the \(1\)-form which pulls back by \(\map{E}{\qE}\) to become \(ω0\) at the point \(e\).
If we replace \((m,e)\) by some point \((m,eh)\), for some \(h \in H\),
\[
r_h^*ω0 = \Ad_h^{-1} ω0.
\]
In particular, if \(h \in G_+\), then \(r_h^*ω0=ω0\).
By definition, we extend \(s^*ω0\) to \(\qE\) by \(\qH\)-equivariance.
Since \(ω0\) is complex linear on tangent spaces of \(E\), \(ωe\) is complex linear on tangent spaces of \(\qE\).
For any \(v \in \LieH\), \(v \hook \qom=\qv\), so \(\qv \hook ω0=\qv\).
\end{proof}

\subsection{Characteristic forms and classes}
The \emph{Atiyah form}, \(k^{\text{th}}\) \emph{Chern form}, \emph{Chern character form}, \emph{Todd form}, etc.\ of a Cartan geometry with Cartan connection \(\omega\) is the form identified by \(\omega\) with the Atiyah form, \(k^{\text{th}}\) Chern form, Chern character form, Todd form, etc.\ and similarly for any \(G_0\)-invariant homogeneous polynomial function \(\map[f]{\LieG_0}{\C}\).
\begin{lemma}
The Atiyah class, \(k^{\text{th}}\) Chern class, and so on, in Dolbeault cohomology of the bundle \(\map{\qE}{M}\) of a Cartan connection \(\map{E}{M}\) with Langlands decomposition is the class of the \((1,1)\)-part, \((k,k)\)-part, and so on, of the pullback by a \(C^{\infty}\) section \(s\) of the Atiyah form.
\end{lemma}
\begin{proof}
The Atiyah class \cite{Atiyah1957} of \(\qE\) is represented by
\[
a(M,\qE)=[\bar\partial ω0]=\left[(dω0)^{1,1}\right].
\]
If we pick any \smooth reduction, the \(2\)-form \((dω0)^{1,1}\) pulls back to \(s^*E\) to become
\[
-\proj[\LieG_0]{\rho_{\LieG}(ω+^{0,1}) ʌω-},
\]
where \(\mapto{w \in \LieG}{\proj[\LieqH]{w}\in\LieqH}\) is the \(G_0\)-invariant linear projection.
The Atiyah class is represented by
\[
(dω0)^{1,1} = -\proj[\LieG_0]{\rho_{\LieG}(ω+^{0,1}) ʌω-}
\in
\nForms[V]{1,1}{M}
\]
where \(V\subset\LieG_0\) is the projection to \(\LieG_0\) of the span of
\[
\rho_{\LieG}(\LieG_+) (\LieG_-) \subset \LieG.
\]
\end{proof}
\begin{example}
If the first Chern form of a complex homogeneous space vanishes, then every complex manifold with a Cartan geometry modelled on that complex homogeneous space has a holomorphic connection on its canonical bundle.
\end{example}
\begin{corollary}
For a complex manifold admitting a holomorphic Cartan geometry, whose model has a Langlands decomposition with \(G_+\) acting trivially on \(\LieG/\LieH\), the Atiyah class of the tangent bundle is the \((1,1)\)-part of the pullback by any \(C^{\infty}\) section of the form identified by the Cartan connection with the tangent bundle Atiyah form.
\end{corollary}
The symmetry of the tangent bundle Atiyah class is well known.
If \(T\subset G_0\) is a Cartan subgroup, then \(T\) has conjugates Zariski dense in \(G_0\).
So the invariant polynomials are determined by their values on \(T\), so expressed in terms of \(T\)-weights.
Relations among the \(T\)-weights acting on \(\LieG_0\) give equations on the Chern classes in Dolbeault cohomology of any holomorphic principal \(G_0\)-bundle.
Our main aim in this paper is to the additional hypothesis of a holomorphic Cartan geometry to improve on those well known equations.
\begin{theorem}\label{theorem:X.G.G.mod}
For any holomorphic \((X,G)\)-geometry, the Atiyah class of any vector bundle associated to any \(\LieG\)-module vanishes, and hence the Chern classes and all of their associated Chern--Simons classes except perhaps \(T_{c_1}\).
\end{theorem}
\begin{proof}
The Atiyah classes are pulled back from \(E_G\), which has the Cartan connection as a holomorphic connection, so all Atiyah classes vanish.
The Chern--Simons classes are \((p+1,p)\) parts of polynomials in the curvature, so vanish except perhaps for \(p=0\), i.e. \(T_{c_1}\).
\end{proof}
Therefore the Atiyah classes of any Cartan geometry modelled on a complex homogeneous space \((X,G)\) lie in the quotient of Atiyah classes of the \(\LieH\)-modules by those of the \(\LieG\)-modules as defined above.
We will adopt a formal cohomology theory below in which the associated formal Atiyah classes of all \(\LieG\)-modules vanish.

\section{Example: projective connections}\label{section:projective.connections}
A \emph{projective connection} is a Cartan geometry modelled on \((X,G)=(\Proj{n},\PSL{n+1})\).
So \(H\subset \PSL{V}\) is the stabilizer of a point of projective space, i.e. the set of matrices (defined up to \(n+1\) root of unity scalar multiple) of the form
\[
\begin{bmatrix}
g^0_0 & g^0_i \\
0 & g^i_j
\end{bmatrix}
\]
with \(i,j=1,2,\dots,n\).
The subgroup \(G_+\subseteq H\) consists of the matrices of the form
\[
\begin{bmatrix}
1 & g^0_i \\
0 & \delta^i_j
\end{bmatrix}.
\]

Recall the conventional notation \cite{Cartan:1992}.
Write the Cartan connection not as \(\omega\) but as \(\Omega\in\Omega^1_E\otimes \LieGL{n+1}\).
Define
\begin{align*}
\omega^i&\defeq\Omega^i_0,\\
\omega^i_j&\defeq\Omega^i_j-\delta^i_j\Omega^0_0,\\
ωi&=\Omega^0_i,
\end{align*}
so that
\begin{align*}
\Omega^0_0&=-\frac{1}{n+1}\omega^i_i,\\
\Omega^i_0&=\omega^i,\\
\Omega^i_j&=\omega^i_j-\frac{\delta^i_j\omega^k_k}{n+1},\\
\Omega^0_i&=ωi.
\end{align*}
In this notation, the structure equations of a projective connection expand out to
\begin{align*}
0&=d\omega^i+\omega^i_j ʌ\omega^j,\\
0&=d\omega^i_j+\omega^i_k\wedge\omega^k_j - (\delta^i_j ωk + \delta^i_k ωj)\wedge\omega^k,\\
0&=dωi-\omega^j_i\wedgeωj
\end{align*}
modulo first and second order torsion, i.e. the curvature of the Cartan connection.

The group \(G=\PSL{n+1}\) acts transitively on all complex bases of tangent spaces of \(\Proj{n}\).
Hence the bundle \(\qE\) is the frame bundle of the manifold \(M\).
Its Atiyah class is represented by
\[
a(TM)=d\omega^i_j+\omega^i_k\wedge\omega^k_j=(\delta^i_j ωk + \delta^i_k ωj)\wedge\omega^k
\]
modulo torsion.
Taking trace gives the first Chern class representation:
\[
2\pi c_1(TM)=(n+1)\sqrt{-1}ωk\wedge\omega^k.
\]

For simplicity of notation, denote \(\LieG_-\) as \(\C[n]\).
So \(\LieG_+=\C[n*]\) and \(\LieG_0=\LieGL{n}\).
Take \(G_0\subset H\) to be the matrices
\[
\begin{bmatrix}
g^0_0 & 0 \\
0 & g^i_j
\end{bmatrix},
\]
giving a Langlands decomposition.
The Atiyah class is represented by an element
\[
a \in \LieG_+^* \otimes\LieG_-^* \otimes \LieG_0
=
\C[n] \otimes \C[n*] \otimes (\C[n] \otimes \C[n*]),
\]
given, for \(x\in\LieG_+=\C[n*]\) and \(y,z\in \C[n]\), by
\[
a(x,y)z=z(x\cdot y)+y(x\cdot z).
\]
The first Chern class in Dolbeault cohomology is represented by the first Chern form
\[
2\pi c_1(x,y)=\sqrt{-1}(n+1)x \cdot y.
\]
Hence
\[
\sqrt{-1}(n+1)a=2\pi I\otimes c_1+2 \pi c_1\otimes I.
\]
We recover:
\begin{theorem}[Molzon and Mortensen \cite{Molzon/Mortensen:1996}]
Take a complex manifold \(M\) with a holomorphic projective connection.
Then, in Dolbeault cohomology,
\[
\sqrt{-1}(1+\dim M)a(TM)=2 \pi I \otimes c_1(TM)+2 \pi c_1(TM)\otimes I.
\]
\end{theorem}
Conversely, Molzon and Mortensen prove that if this condition is satisfied, then the complex manifold admits a holomorphic projective connection, which can be chosen to be normal.
\begin{corollary}
A holomorphic projective connection on a complex manifold arises from a holomorphic affine connection just when the complex manifold has \(c_1=0\) in Dolbeault cohomology.
\end{corollary}
The bundle \(\map{\OO(1)}{\Proj{n}}\) is acted on by \(\SL{n+1}\), but \emph{not} by \(\PSL{n+1}\).
Let \(\hat{H}\subset \SL{n+1}\) be the set of matrices of the form
\[
g=
\begin{pmatrix}
g^0_0 & g^0_i \\
0 & g^i_j
\end{pmatrix}
\]
and let
\[
\rho(g)=\frac{1}{g^0_0}.
\]
For any integer \(d\), let \(\C_d\) be \(\C\) but with \(\hat{H}\)-module structure imposed by \(\rho_d\defeq \rho^d\).
Let \(\OO(d)\) be the associated line bundle \(\OO(d)=\defeq\SL{n+1} \times^{\rho}\C_d\).
The group \(H\subset\PSL{n+1}\) is the image of \(\hat{H}\subset\SL{n+1}\), i.e. \(\hat{H}\) modulo \((n+1)\)-roots of unity.
Hence \(\C_d\) is an \(H\)-module just when \(n+1\) divides \(d\).
The subalgebra of \(\LieH\) acting trivially on \(\C_d\) contains \(\LieG_+\), so \(\C_d\) is an \(\LieqH=\LieH/\LieG_+\)-module.
For a projective connection,  when \(n+1\) does not divide \(d\), \(\OO(d)\) is not an associated bundle.
By our general theory of characteristic classes, it still has an Atiyah class \(a(\OO(d))\), and Chern classes \(c_k(\OO(d))\).
Hence it is convenient to write this \(\LieqH\)-module \(\C_d\) also as \(\OO(d)\).

Recall \(V\defeq\C[n+1]\) is the obvious \(\LieG\)-module as \(\LieG=\LieSL{n+1}\); note that this is not a trivial representation.
%The Atiyah class of \(V\) is
%\[
%a_V(x,y)=
%\begin{pmatrix}
%-x_iy^i&0\\
%0&y^ix_j
%\end{pmatrix}
%\]
%which is not zero.
By \vref{theorem:X.G.G.mod}, the Atiyah class of the vector bundle \(\vb{V}\) associated to \(V\) vanishes on any projective connection geometry.
So we can work modulo this Atiyah class.
The theory above describing Atiyah forms does not apply to \(\vb{V}\), as it has nontrivial \(G_+\)-action, so is not defined on \(\qE\).

Note that \(\OO(-1)\subset V\) is an \(\hat{H}\)-submodule and tensoring gives \(\OO\subset V\otimes\OO(1)\).
Recall the Euler sequence on \(\Proj{n}\):
\[
0 \to \OO \to V\otimes \OO(1) \to T \to 0
\]
where \(T=\LieG/\LieH\).

Let us rewrite this directly in terms of a projective connection on a manifold:
\begin{align*}
a_{V}(M)
&=
\rho_V(\nabla\omega),
\\
&=
\nabla\Omega-
\begin{pmatrix}
0&0\\
\nabla\omega^i&0
\end{pmatrix},
\\
&=
0-0
\end{align*}
up to torsion terms.

Let us consider this vanishing more geometrically.
Locally (say on a cover by open sets \(M_a\subset M\)), we can lift any \(H\)-bundle \(\map{E}{M}\) to a \(\hat{H}\)-bundle \(\map{E_a}{M_a}\) and in doing so lift the projective connection to a holomorphic Cartan geometry modelled on \((\Proj{n},\hat{G})\) where \(\hat{G}\defeq\SL{n+1}\).
So \(\Omega\) is a holomorphic connection for the associated vector bundle \(\amal{E_a}{\hat{H}}{\C[n+1]}\).
Indeed it suffices to note that \(V\) is a \(\hat{G}\)-module, and so the Cartan connection extends to a holomorphic connection on the associated \(\hat{G}\)-bundle.
So the Atiyah class vanishes: \(a_V=0\) on each \(E_a\).
But the Atiyah classes in these calculations, as ghost vector bundle Atiyah classes, are defined on \(E\) and pull back to each \(E_a\) to have this same representation as a differential form, and so \(a_V=0\) on \(E\).

Apply the splitting principle to the Euler sequence:
\[
c(V\otimes\OO(1))=c(\OO(1)[n+1])=(1+c_1\OO(1))^{n+1}=c(T)
\]
in Dolbeault cohomology and therefore
\[
(n+1)c_1(\OO(1))=c_1(T).
\]
\begin{theorem}
On any \(n\)-dimensional complex manifold with holomorphic projective connection,
\[
(n+1)^kc_k(T)=\binom{n+1}{k}c_1^k(T)
\]
in Chern forms modulo torsion and curvature, and hence also as Chern classes in Dolbeault cohomology.
Equivalently
\[
1+\ch{}=(n+1)e^{c_1/(n+1)}
\]
in Chern forms modulo torsion and curvature, and hence also as Chern classes in Dolbeault cohomology.
\end{theorem}
We can compute this more directly in the structure equations, using the Chern character, as
\[
(-2\pi i)c_1 = (n+1)ωi\wedge\omega^i.
\]
(Note that we actually compute such things formally, as polynomials in the formal Atiyah class, but we can do so as if we are working with the differential forms of the structure equations.)
It helps to write \(\omega\indices{^i^j}\) to mean \(\omega^i\wedge\omega^j\), \(\omega\indices{_i^j}\) to mean \(ωi\wedge\omega^j\), and so on.
Similarly \(a^{ij}_{k\ell}\) means \(a^i_ja^k_{\ell}\) and so on.
Set \(\ch{j}\defeq \ch{j}[T]\) to be the Chern character components,
\begin{align*}
(-2\pi i)^j j!\ch{j}
&=
a^{i_1 i_2 \dots i_{j-1} i_j}_{i_2 i_3 \dots i_{j\phantom{-1}} i_1},\\
&=
(\delta^{i_1}_{i_2}\omega\indices{_k^k}+\omega\indices{_{i_2}^{i_1}})
\dots
(\delta^{i_j}_{i_1}\omega\indices{_k^k}+\omega\indices{_{i_1}^{i_j}}),\\
&=
n(\omega\indices{_k^k})^j\\
&\qquad+
\sum_{p=0}^{j-1}
\binom{j}{p}
(\omega\indices{_k^k})^p
\omega^{\phantom{i_2}i_1\phantom{i_3}i_2\dots\phantom{i_1}i_{j-p}}_{i_2\phantom{i_1}i_3\phantom{i_2}\dots i_1\phantom{i_{j-p}}},\\
&=
n(\omega\indices{_k^k})^j\\
&\qquad+
\sum_{p=0}^{j-1}
\binom{j}{p}
(\omega\indices{_k^k})^p(-1)^{j-p}
\omega^{i_1\phantom{i_2}i_2\phantom{i_3}\dots i_{j-p}}_{\phantom{i_1}i_2\phantom{i_1}i_3\dots \phantom{i_{j-p}}i_1},\\
%\omega^{i_1}\wedgeω{i_2}\wedge\omega^{i_2}\wedgeω{i_3}\wedge\dots\wedgeω{i_{j-p}}\wedge\omega^{i_{j-p}}\wedgeω{i_1},\\
&=
n(\omega\indices{_k^k})^j\\
&+\qquad
\sum_{p=0}^{j-1}
\binom{j}{p}
(\omega\indices{_k^k})^p(-1)^{j-p+1}
\omega^{\phantom{i_1}i_1\phantom{i_2}i_2\dots\phantom{i_{j-p}}i_{j-p}}_{i_1\phantom{i_1}i_2\phantom{i_2}\dots i_{j-p}\phantom{i_{j-p}}},\\
%ω{i_1}\wedge\omega^{i_1}\wedgeω{i_2}\wedge\omega^{i_2}\wedgeω{i_3}\wedge\dots\wedgeω{i_{j-p}}\wedge\omega^{i_{j-p}},\\
&=
\left(
n
+
\sum_{p=0}^{j-1}
\binom{j}{p}
(-1)^{j-p+1}
\right)(\omega\indices{_k^k})^j
,\\
&=
\left(
n
+
1\right)
(\omega\indices{_k^k})^j
\end{align*}
so that
\[
\frac{\ch{j}}{n+1}=\frac{\left(\frac{\ch{1}}{n+1}\right)^j}{j!}
\]
for \(j=1,2,\dots\).
So finally
\[
1+\ch{}=(n+1)e^{c_1/(n+1)}.
\]
\subsection{Chern--Simons forms}
As a symmetric polynomial on \(\LieqH\), let
\[
\chp{j}\defeq \ch{j}-\frac{(n+1)}{j!}\left(\frac{c_1}{n+1}\right)^j.
\]
We have found above that \(\chp{j}=0\) in Dolbeault cohomology on any complex manifold with holomorphic projective connection.
Hence \(T_{\chp{j}}\) drops to define a class in Dolbeault cohomology, the associated Chern--Simons class, on any complex manifold with a holomorphic projective connection.
\begin{theorem}
If a complex manifold \(M\) of complex dimension \(n\) admits a holomorphic projective connection then in Dolbeault cohomology \(0=\chp{2}=\dots=\chp{n}\) and
\[
\sum_{j=2}^n \binom{j}{2}T_{\chp{j}}
=
T_{\chp{2}} ʌe^{c_1/(n+1)}.
\]
\end{theorem}
\begin{proof}
Note that the Atiyah class of a projective connection is
\[
a^i_j=\delta^i_jωk\wedge\omega^k+ωj\wedge\omega^i.
\]
Hence the Chern--Simons invariant is
\begin{align*}
j(-2\pi i)^j j! \CS{\ch{j}}
&=
\omega^{i_1}_{i_2}ʌa^{i_2}_{i_3} ʌ\dots ʌa^{i_j}_{i_1}\\
&\qquad+
a^{i_1}_{i_2}ʌ\omega^{i_2}_{i_3} ʌ\dots ʌa^{i_j}_{i_1}\\
&\qquad
+
\dots\\
&\qquad
+
a^{i_1}_{i_2}ʌa^{i_2}_{i_3} ʌ\dots ʌ\omega^{i_j}_{i_1},
\\
&=
\omega^{i_1}_{i_2}ʌa^{i_2}_{i_3} ʌ\dots ʌa^{i_j}_{i_1}\\
&\qquad
+
\omega^{i_2}_{i_3} ʌ a^{i_3}_{i_4} ʌ\dots ʌa^{i_j}_{i_1}ʌa^{i_1}_{i_2}\\
&\qquad
+
\dots\\
&\qquad
+
\omega^{i_j}_{i_1} ʌa^{i_1}_{i_2}ʌa^{i_2}_{i_3} ʌ\dots ʌa^{i_{j-1}}_{i_j},
\\
&=
j\omega^{i_1}_{i_2}ʌa^{i_2}_{i_3} ʌ\dots ʌa^{i_j}_{i_1}.
\end{align*}
Sum over over all ways to pick \(p\) of the Atiyah class expressions (of which there are \(j-1\) in all) to expand out as \(\delta^{i_{\ell-1}}_{i_{\ell}}ωk\wedge\omega^k\) and \(j-1-p\) to expand out as \(ω{i_{\ell}}\wedge\omega^{i_{\ell-1}}\):
\begin{align*}
(-2\pi i)^j j!\CS{\ch{j}}
&=
\sum_{p=0}^{j-1}\binom{j-1}{p}(\omega\indices{_k^k})^p
ʌ\omega^{i_1\phantom{i_3} i_2 \dots \phantom{i_1} i_{j-p}}_{i_2 i_3 \phantom{i_2} \dots i_1\phantom{i_{j-p}}}
\\
&=
(\omega\indices{_k^k})^{j-1}ʌ\omega^i_i
+
\sum_{p=0}^{j-2}\binom{j-1}{p}(-1)^{j-p-1}(\omega\indices{_k^k})^p
ʌ\omega^{i_1i_2\phantom{i_3}  \dots i_{j-p}\phantom{i_1} }_{i_2 \phantom{i_2} i_3 \dots \phantom{i_{j-p}}i_1}
\\
&=
(\omega\indices{_k^k})^{j-1}ʌ\omega^i_i
+
\left(\sum_{p=0}^{j-2}\binom{j-1}{p}(-1)^{j-p-1}\right)(\omega\indices{_k^k})^{j-2}
ʌ\omega^{i_1i_2\phantom{i_3}}_{i_2\phantom{i_2}i_1}
\\
&=
(\omega\indices{_k^k})^{j-2}
ʌ(\omega^{i\phantom{k}k}_{ik\phantom{k}}+\omega^{i\phantom{i}\ell}_{\ell i\phantom{\ell}})
\\
&=
\left(\frac{-2\pi i c_1}{n+1}\right)^{j-2}2(-2\pi i)^2\CS{\ch{2}}.
\end{align*}
Dividing off \(-2\pi i\) factors:
\[
j!T_{\ch{j}}
=
2T_{\ch{2}}\left(\frac{c_1}{n+1}\right)^{j-2},
\]
which we can write as
\[
\binom{j}{2}T_{\ch{j}}
=
T_{\ch{2}}ʌ\frac{1}{(j-2)!}\left(\frac{c_1}{n+1}\right)^{j-2}.
\]
Similarly
\[
T_{\ch{1}^j} = T_{c_1}ʌc_1^{j-1},
\]
from which the result follows.
\end{proof}

\section{\texorpdfstring{Example: $1$-flat foliations}{Example: 1-flat foliations}}
Consider a complex manifold \(M\) bearing a holomorphic projective connection and a holomorphic subbundle of the tangent bundle.
As a model, fix a projective linear subspace \(\Proj{p-1} \subset \Proj{p+q}\) and let \(X=\Proj{p+q}-\Proj{p-1}\).
Foliate by taking all \(p\)-dimensional projective subspaces containing \(\Proj{p-1}\).
Write \(\Proj{p-1}\) as the projectivization of the linear subspace of vectors
\[
\begin{pmatrix}
0\\
x^i\\
0
\end{pmatrix}.
\]
The group \(G\) of all projective transformations preserving \(\Proj{p-1}\) acts transitively on \(X\) fixing that foliation, consisting of elements
\[
\begin{bmatrix}
g^0_0&0&g^0_J\\
g^i_0&g^i_j&g^i_J\\
g^I_0&0&g^I_J
\end{bmatrix}\in\PSL{p+q+1}.
\]
The reader familiar with the method of equivalence will see that this model emerges from application of the method \cite{Gardner:1989}, being the unique example of maximal dimensional symmetry group of a holomorphic projective connection with a holomorphic subbundle of the tangent bundle.
We take the projective connection \(\map{E_1}{M}\), and consider the subbundle \(E_2\subset E_1\) on which the soldering form identifies elements of the given vector subbundle \(V\subset TM\) with elements of \(\C[p]\oplus 0\subset \C[p+q]\).
On this subbundle, \(\omega^I_i\) vanishes on the vertical directions, so is semibasic, say \(\omega^I_i=t^I_{ij}\omega^j+t^I_{iJ}\omega^J\).
The antisymmetric part \(t^I_{ij}-t^I_{ji}\) is the obstruction to \(V\) being a foliation.
The symmetric part remaining is the shape operator; \cite{Rovenskii:1998} p. 31.
The reduction of structure group following the method of equivalence is not of constant type in general, so to ensure constant type we insist that will only study totally geodesic holomorphic foliations, i.e. \(t^I_{ij}=0\).

Differentiating the equation \(\omega^I_i=t^I_{ij}\omega^j+t^I_{iJ}\omega^J\), we find that \(dt^I_{iI} = qωi+t^I_{jI}\omega^j_i\) modulo semibasic terms.
The set of points \(E\subset E_1\) at which \(t^I_{iI}=0\) is a principal \(H\)-subbundle, where \(H\) is the set of all elements of \(G\) of the form
\[
\begin{bmatrix}
g^0_0 &     0 & g^0_J \\
    0 & g^i_j & g^i_J \\
    0 &     0 & g^I_J
\end{bmatrix}.
\]
Consider the subgroup \(G_0\) of elements of the form
\[
\begin{bmatrix}
g^0_0 &     0 & 0 \\
    0 & g^i_j & 0 \\
    0 &     0 & g^I_J
\end{bmatrix}
\]
and the subgroup \(G_+\) of elements of the form
\[
\begin{bmatrix}
    1 &     0 & g^0_J \\
    0 & \delta^i_j & g^i_J \\
    0 &     0 & \delta^I_J
\end{bmatrix}.
\]
Since our model has no further reduction of structure group, the method of equivalence finishes at \(E\).
We are inclined to pretend that this is a Cartan geometry with Cartan connection
\[
\omega=ω-\oplusω0\oplusω+
=
(\omega^i,\omega^I)\oplus(\omega^i_j,\omega^I_J)\oplus(ωI)
\]
with model \((X,G)\) as above.
It might not be a Cartan connection, because the curvature might not be semibasic.

We compute that
\begin{align*}
d\omega^i_j+\omega^i_k\wedge\omega^k_j
&=
\delta^i_jωK\wedge\omega^K
-t^J_{jk}\omega^i_J\wedge\omega^k
-t^J_{jK}\omega^i_J\wedge\omega^K,\\
d\omega^I_J+\omega^I_K\wedge\omega^K_J
&=
\delta^I_JωK\wedge\omega^K
+t^I_{kL}\omega^k_J\wedge\omega^L,\\
\end{align*}
modulo semibasic terms, i.e. modulo the various \(\omega^k\wedge\omega^{\ell}\).
We see that our theory will not allow to compute equations on the characteristic classes of the \(H/G_+\)-bundle unless \(0=t^I_{jk}=t^I_{jK}\), i.e. the vector subbundle \(V\subset TM\) is a foliation \(F\) with totally geodesic leaves and also the invariant \(t^I_{iJ}\), a \(1\)-form on the leaves of the foliation valued in traceless endomorphisms of the normal bundle, must vanish.
Such foliations we call \emph{\(1\)-flat} for the projective connection.

Every \(1\)-flat foliation is a Cartan geometry modelled on \((X,G)\), with
\begin{align*}
d
\begin{pmatrix}
\omega^i \\
\omega^I
\end{pmatrix}
&=
-
\begin{pmatrix}
\omega^i_j & \omega^i_J \\
0 & \omega^I_J
\end{pmatrix}
\wedge
\begin{pmatrix}
\omega^j \\
\omega^J
\end{pmatrix},
\\
d
\begin{pmatrix}
\omega^i_j & \omega^i_J \\
0 & \omega^I_J
\end{pmatrix}
&=
-
\begin{pmatrix}
\omega^i_k & \omega^i_K \\
0 & \omega^I_K
\end{pmatrix}
\wedge
\begin{pmatrix}
\omega^k_j & \omega^k_J \\
0 & \omega^K_J
\end{pmatrix}
\\
&\qquad+
\begin{pmatrix}
\delta^i_j ωK ʌ\omega^K & ωJ ʌ\omega^i \\
0 & (ωJ \delta^I_K + ωK \delta^I_J)\wedge\omega^K
\end{pmatrix}
\end{align*}
modulo semibasic terms.

Since a \(1\)-flat foliation occurs on a complex manifold with holomorphic projective connection, it has the same characteristic class and Chern--Simons equations we found before.
But we will also find some new ones.
The total Chern class of the tangent bundle \(TF\) of the foliation \(F\) in Dolbeault cohomology is
\[
c(TF)=\det\of{I-\frac{\sqrt{-1}}{2 \pi} \delta^i_j ωK ʌ\omega^K}
=\pr{1-\frac{\sqrt{-1}}{2 \pi} ωK ʌ\omega^K}^p
\]
so that
\[
c_k(TF)
=
\binom{p}{k}c_1(TF)^k
\]
for \(k=1,2,\dots,q\).
We also see that
\[
c_1(TF)=-\frac{i}{2\pi}pωK ʌ\omega^K
\]
and
\[
c_1(TM/TF)=-(q+1)\frac{i}{2\pi}ωK ʌ\omega^K
\]
so that
\[
p \, c_1(TM/TF) = (q+1) \, c_1(TF),
\]
and so
\[
c_1(TM/TF)^{q+1}=c_1(TF)^{q+1}=0.
\]
In terms of \(c_1(TM)\),
\[
c_1(TM)=c_1(TF)+c_1(TM/TF),
\]
giving
\begin{align*}
c_k(TF)&=\binom{p}{k}\left(\frac{p}{p+q+1}c_1(TM)\right)^k,\\
c_k(TM/TF)&=\binom{q+1}{k}\left(\frac{q+1}{p+q+1}c_1(TM)\right)^k.
\end{align*}

Take any \(\GL{q}\)-invariant polynomial \(P\) of degree \(\ge q+1\), perhaps valued in a finite dimensional holomorphic \(\GL{q}\)-module.
Write, as above,
\[
\nabla \omega^I_J = d\omega^I_J + \omega^I_K ʌ\omega^K_J.
\]
We find
\[
P(\nabla \omega^I_J)=P(\delta^I_J ωK ʌ\omega^K+ωJ\wedge\omega^I),
\]
modulo torsion, expands out to have more than \(q\) \(1\)-forms \(\omega^K\) in each term.
But there are only \(q\) such \(1\)-forms, so 
\[
P(\nabla \omega^I_J)=0
\]
modulo torsion.
The Chern--Simons form is then
\[
\CS{P}=P(\omega^I_J,\nabla \omega^I_J,\dots,\nabla \omega^I_J)
\]
which also vanishes, modulo torsion, if there are more than \(q\) \(1\)-forms \(\omega^K\) in each term, i.e. if \(P\) has degree \(q+2\) or more.
We recover the Baum--Bott theorem \cite{Baum/Bott:1972} p. 287 for our holomorphic folation, with results of Kamber and Tondeur \cite{Kamber/Tondeur:1975}:
\begin{theorem}
Take a complex manifold with a holomorphic projective connection and a \(1\)-flat foliation.
All Chern classes, in Dolbeault cohomology, of the normal bundle of the foliation, of degree more than the codimension of the foliation, vanish.
All of their associated Chern--Simons classes, in Dolbeault cohomology, 
of degree at least two more than the codimension of the foliation, vanish.
\end{theorem}
Again, we stress that this theorem is a direct consequence of the linear algebra computation of \(\lb{\LieG_+}{\LieG_-}_0\) for \(G \subset \GL{p+q}\) the stabilizer of a \(p\)-dimensional linear subspace.

From the structure equations, we see that:
\begin{theorem}
Take a complex manifold \(M\) equipped with a holomorphic projective connection and a \(1\)-flat totally geodesic foliation.
On each leaf, the tangent bundle of the leaf admits a holomorphic affine connection, as does the normal bundle of the leaf.
Hence all characteristic classes of those bundles restrict to each leaf to vanish in Dolbeault cohomology.
\end{theorem}

\section{Example: projective connections and split tangent bundle}
Take a complex manifold \(M\) with a holomorphic projective connection \(\map{E}{M}\) and split tangent bundle \(TM=U\oplus V\).
Let \(E'\subset E\) be the set of points at which the soldering form identifies the splitting with \(\C[p] \oplus \C[q] = \C[p+q]\).
Hence \(\omega^I_i\) and \(\omega^i_I\) both become semibasic on \(E'\).
As above, \(dt^I_{iI}=pωi+t^I_{jI}\omega^j_i\) and \(dt^i_{Ii}=qωI+t^i_{Ji}\omega^J_I\).
We let \(E''\subset E\) be the set of points at which \(0=t^I_{iI}\) and \(0=t^i_{Ii}\).
The structure equations of a projective connection (modulo torsion and curvature) are
\begin{align*}
  d \omega^i + \omega^i_j ʌ\omega^j + \omega^i_J ʌ\omega^J &=0,\\
  d \omega^I + \omega^I_j ʌ\omega^j + \omega^I_J ʌ\omega^J &=0, \\
  d \omega^i_j + \omega^i_k ʌ\omega^k_j +\omega^i_K\wedge\omega^K_j &=
  \pr{\delta^i_j ωk + \delta^i_k ωj} ʌ\omega^k+
  \delta^i_j ωK ʌ\omega^K,\\
  d \omega^i_J + \omega^i_k ʌ\omega^k_J +\omega^i_K\wedge\omega^K_J &=
  ωJ ʌ\omega^i,\\
  d \omega^I_j + \omega^I_k ʌ\omega^k_j +\omega^I_K\wedge\omega^K_j &=
  ωj ʌ\omega^I,\\
  d \omega^I_J + \omega^I_k ʌ\omega^k_J +\omega^I_K\wedge\omega^K_J &=
  \delta^I_J ωk ʌ\omega^k+
  (\delta^I_J ωK + \delta^I_K ωJ)ʌ\omega^K,\\
  d ωi &= \omega^j_i ʌωj+\omega^J_i ʌωJ,\\
  d ωI &= \omega^j_I ʌωj+\omega^J_I ʌωJ.
\end{align*}
These reduce on \(E''\), modulo torsion, to
\begin{align*}
  d \omega^i + \omega^i_j ʌ\omega^j &=0,\\
  d \omega^I + \omega^I_J ʌ\omega^J &=0, \\
  d \omega^i_j + \omega^i_k ʌ\omega^k_j &=0,\\
  d \omega^I_J + \omega^I_K\wedge\omega^K_J &= 0, \\
\end{align*}
which are the structure equations of Cartan geometry modelled on a product of affine spaces, each acted on by the affine group.
The Atiyah class of the tangent bundle vanishes.
\begin{theorem}
On a complex manifold, a choice of holomorphic projective connection and a splitting of its tangent bundle into a direct sum imposes an affine connection on the tangent bundle.
In particular, all characteristic classes of the tangent bundle vanish in Dolbeault cohomology.
\end{theorem}
The generalizes a result of Kobayashi and Ochiai \cite{KobayashiOchiai:1980} p. 85 theorem 6.6; their result requires that the splitting be into tangent bundles of foliations.

\section{Example: Grassmannian geometries}
A \emph{Grassmannian geometry} is a Cartan geometry modelled on the complex homogeneous space
\[
(X,G)=(\Gr{p}{\C[p+q]},\PSL{p+q}).
\]
The structure equations:
\[
0
=
d
\begin{pmatrix}
\omega^i_j & \omega^i_J\\
\omega^I_j & \omega^I_J
\end{pmatrix}
+
\begin{pmatrix}
\omega^i_k & \omega^i_K\\
\omega^I_k & \omega^I_K
\end{pmatrix}
\wedge
\begin{pmatrix}
\omega^k_j & \omega^k_J\\
\omega^K_j & \omega^K_J
\end{pmatrix}
\]
modulo torsion, \(i,j,k=1,2,\dots,p\), \(I,J,K=p+1,\dots,p+q\).
Split into the obvious Langlands decomposition:
\[
ω-=
\begin{pmatrix}
0 & 0 \\
\omega^I_j & 0
\end{pmatrix},
ω0
=
\begin{pmatrix}
\omega^i_j & 0\\
0 & \omega^I_J
\end{pmatrix},
ω+=
\begin{pmatrix}
0 & \omega^i_J\\
0 & 0
\end{pmatrix}
\]
The tangent bundle is a tensor product of two bundles, which have Atiyah class representatives
\begin{align*}
d\omega^i_j+\omega^i_k\wedge\omega^k_j&=-\omega^i_K\wedge\omega^K_j,\\
d\omega^I_J+\omega^I_K\wedge\omega^K_J&=+\omega^k_J\wedge\omega^I_k,\\
\end{align*}
modulo curvature.
Differentiating the soldering forms
\[
d\omega^I_i=-\omega^{Ij}_{iJ}\wedge\omega^J_j
\]
modulo torsion, where
\[
\omega^{Ij}_{iJ}\defeq\delta^j_i\omega^I_J-\delta^I_J\omega^j_i.
\]
So the Atiyah class of the tangent bundle is represented by
\begin{align*}
a&=d\omega^{Ij}_{iJ} + \omega^{Ik}_{iK} ʌ\omega^{Kj}_{kJ},\\
&=\delta^i_j \omega^k_J ʌ\omega^I_k + \delta^I_J \omega^j_K\wedge\omega^K_i
\end{align*}
modulo curvature.

Define \(\LieH\)-modules \(U=\C[p]\) and \(Q=\C[q]\) by, for
\[
v=
\begin{pmatrix}
v^i_j & v^i_J \\
0 & v^I_J
\end{pmatrix} \in \LieH,
\]
\begin{align*}
\rho_U(v)&=
\begin{pmatrix}
v^i_j
\end{pmatrix},\\
\rho_Q(v)&=
\begin{pmatrix}
v^I_J
\end{pmatrix}
\end{align*}
giving an exact sequence
\[
0 \to U \to \C[p+q] \to Q \to 0
\]
of \(\LieH\)-modules.
In terms of Lie algebras, \(\LieG_+=U \otimes Q^*\), say, and \(\LieG_-=Q\otimes U^*\).
Elements \(x \in U\otimes Q^*\) and \(y \in Q \otimes U^*\), thought of as linear maps \(\map[x]{Q}{U}\) and \(\map[y]{U}{Q}\), have compositions \(xy \in U^* \otimes U\) and \(yx \in Q^*\otimes Q\).
The Atiyah class of the tangent bundle is represented by
\[
a(x,y_1)y_2=-y_1xy_2+y_2xy_1.
\]
The same expression holds identically for the Lagrangian Grassmannian, except that \(x,y_1,y_2\) are symmetric square matrices.

Note that \(\C[p+q]\) is an \(\LieSL{p+q}\)-module so has vanishing Atiyah class.
Clearly \(U\) and \(Q\) descend to \(\LieqH\)-modules.
The \(\LieqH\)-module \(T\defeq U^* \otimes Q\) is an \(\qH\)-module, and the associated vector bundle on any Grassmannian geometry is the tangent bundle.
Neither \(U\) nor \(Q\) are \(\qH\)-modules.
Nonetheless, as previously we write \(a(\vb{U})\) and \(a(\vb{Q})\) for their Atiyah forms.
The splitting principle gives
\[
1=c(\vb{U})c(\vb{Q})
\]
for total Chern forms, while
\[
\operatorname{ch}(\vb{T})=\operatorname{ch}(\vb{U})\operatorname{ch}(\vb{Q})
\]
for total Chern character forms.
Hence on any complex manifold with Grassmann geometry, the characteristic classes in Dolbeault cohomology have similar factorizations to those of the Grassmannian.

\section{Example: conformal geometries}
Take the quadratic form
\[
q(x,y)\defeq \transpose{x}y+\transpose{y}x,
\]
for \(x,y\in\C[n+1]\).
Write vectors as \(x=(x_0,x_1,\dots,x_{n+1})\).
The group \(G=\PO{q}=\PO{2n+2}\) of invertible matrices preserving \(q\) up to factor, modulo rescaling, has Lie algebra \(\LieG=\LieSO{2n+2}\) consisting of the matrices of the form
\[
\begin{pmatrix}
a&b\\
c&-\transpose{a}
\end{pmatrix}
\]
with \(0=\transpose{b}+b=\transpose{c}+c\).
Split up each such matrix into
\[
\begin{pmatrix}
a_{00} & a_{0j} &     0  & b_{0j} \\
a_{i0} & a_{ij} & b_{i0} & b_{ij} \\
    0  & c_{0j} & -a_{00} & -a_{j0} \\
c_{i0} & c_{ij} & -a_{0i} & -a_{ji}
\end{pmatrix}
\]
The group \(G\) acts on the quadric hypersurface \(X=X^{2n}=(q(x,y)=0)\subset\Proj{2n+1}\), i.e. the set of all null lines, with the element \(x_0=[1,0,\dots,0]\) having stabilizer \(H\) with Lie algebra \(\LieH\) given by
\[
\begin{pmatrix}
a_{00}& a_{0j} & 0          & b_{0j} \\
0     & a_{ij} & b_{i0}     & b_{ij} \\
0     & 0      & -a_{00}    & 0 \\
0     & c_{ij} & -a_{0i}   & -a_{ji}
\end{pmatrix}
\]
with \(0=b_{ij}+b_{ji}=c_{ij}+c_{ji}\).
It is convenient to write this as
\[
\begin{pmatrix}
a     & a_j    & 0    & -b_j \\
0     & a_{ij} & b_i  & b_{ij} \\
0     & 0      & -a   & 0 \\
0     & c_{ij} & -a_i & -a_{ji}
\end{pmatrix}
\]
The space \(\LieG_-=\LieG/\LieH\) consists of the elements of the form
\[
(x,y)\cong
\begin{pmatrix}
0     & 0    & 0 & 0 \\
x_i   & 0    & 0 & 0 \\
0     & -y_j & 0 & -x_j \\
y_i   & 0    & 0 & 0
\end{pmatrix}.
\]
So \(\LieH\) acts on \(\LieG_-\) as
\[
\mapto{%%
\begin{pmatrix}
x_i \\
y_i
\end{pmatrix}
}%%
{%%
\begin{pmatrix}
a_{ij} - \delta_{ij} a & b_{ij} \\
c_{ij}                 & -a_{ji}-\delta_{ij} a
\end{pmatrix}
\begin{pmatrix}
x_j \\
y_j
\end{pmatrix}%
}%%
,
\]
i.e.
\[
\rho_{\LieG_-}
\begin{pmatrix}
a     & a_j    & 0      & -b_j \\
0     & a_{ij} & b_i    & b_{ij} \\
0     & 0      & -a     & 0 \\
0     & c_{ij} & -a_i   & -a_{ji}
\end{pmatrix}
=
\begin{pmatrix}
a_{ij} - \delta^i_j a & b_{ij} \\
c_{ij}                & -a_{ji}-\delta_{ij} a
\end{pmatrix}.
\]
Hence the subalgebra \(\LieG_+\) acting trivially on the tangent space \(T_{x_0} X\) consists of the elements of the form
\[
v=
\begin{pmatrix}
0     & a_j & 0      & -b_j \\
0     & 0   & b_i    & 0     \\
0     & 0   & 0      & 0     \\
0     & 0   & -a_i   & 0
\end{pmatrix}.
\]
Bracketing these two gives
\[
\begin{pmatrix}
a_kx_k-b_ky_k & 0              & 0             & 0 \\
0             & -b_iy_j-x_ia_j & 0             & x_ib_j-b_ix_j \\
0             & 0              & y_kb_k-x_ka_k & 0 \\
0             & a_iy_j-y_ia_j  & 0             & a_i x_j + b_j y_i
\end{pmatrix}.
\]
Applying \(\rho_{\LieG_-}\) gives a representative for the Atiyah class of the tangent bundle:
\[
\rho_{\LieG_-}(v)(x,y)=
\begin{pmatrix}
-b_iy_j-x_ia_j-\delta_{ij}(a_kx_k-b_ky_k)&x_ib_j-b_ix_j\\
a_iy_j-y_ia_j&y_ib_j+a_ix_j-\delta_{ij}(a_kx_k-b_ky_k)
\end{pmatrix}.
\]
Taking trace, the first Chern class of the tangent bundle is represented by
\[
-2n(a_kx_k-b_ky_k)=-2nq((a,-b),(x,y)).
\]
Expand out to see that
\begin{align*}
\rho_{\LieG_-}(v)(x_1,y_1)(x_2,y_2)
=&
-q((a,-b),(x_1,y_1))(x_2,y_2)\\
&-q((a,-b),(x_2,y_2))(x_1,y_1)\\
&-q((x_1,y_1),(x_2,y_2))(b,-a).
\end{align*}
The same result holds for odd dimensional conformal geometries, with the same proof.
Hence we recover
\begin{theorem}[\cite{Jahnke/Radloff:2018} p. 6 proposition 2.2]
The Atiyah class of any holomorphic conformal geometry on any complex manifold \(M\) satisfies
\[
i a(T)\dim M=2 \pi c_1(T)\otimes I
+ 2\pi I \otimes c_1(T)
+ 2\pi q \otimes q^* c_1(T)
\]
in Dolbeault cohomology.
\end{theorem}
A \emph{holomorphic Riemannian metric} is a holomorphic Cartan geometry modelled on \((X,G)=(\C[n],\Orth{n} \ltimes \C[n])\).
\begin{corollary}
A holomorphic conformal geometry arises from a holomorphic Riemannian geometry just when \(c_1=0\) in Dolbeault cohomology.
\end{corollary}

Write a quadric \(Q^n\subset\Proj{n+1}\) as \(Q=G/H\) as above, with \(G=\PO{q}\), and also as \(\hat{G}/\hat{H}\), where \(\hat{G}=\Orth{q}\) is the orthogonal group of the quadratic form \(q\).
Any smooth quadric hypersurface \(Q^n\subset\Proj{n+1}\) has exact sequence
\[
0 \to TQ \to \left.T\Proj{n+1}\right|_Q \to N_{Q/\Proj{n+1}} \to 0
\]
and we have normal bundle
\[
N_{Q/\Proj{n+1}}=\left.\OO(2)\right|_Q.
\]
The bundles \(\left.\OO(2)\right|_Q,\left.T\Proj{n+1}\right|_Q\) are the associated vector bundles of the \(H\)-representations taking each
\[
g=
\begin{bmatrix}
g^0_0 & g^0_J\\
0 & g^I_J
\end{bmatrix}\in H
\]
to
\[
\rho_{\OO(2)}
(g)=(g^0_0)^{-2},
\rho_{T\Proj{n+1}}
(g)=(g^0_0)^{-1} g^I_J.
\]
The bundle
\(
\left.T\Proj{n+1}\right|_Q
\)
has exact Euler sequence
\[
0 \to \OO \to \C[n+2]\otimes\OO(1) \to T\Proj{n+1} \to 0.
\]
The line bundle \(\OO(1)\) is \emph{not} an associated vector bundle for \(H\), but of the \(\hat{H}\)-representation
\[
\mapto{g=
\begin{pmatrix}
g^0_0 & g^0_J\\
0 & g^I_J
\end{pmatrix}\in \Orth{q}
}{(g^0_0)^{-1}}.
\]
Nonetheless, as above, every \(\hat{H}\)-module determines a unique \(\LieH\)-module, since \(\map{\hat{H}}{H}\) is an isomorphism of Lie algebras.
Every \(\LieH\)-module has an Atiyah class on any complex manifold with a holomorphic conformal geometry.
The \(\LieH\)-module \(\C[n+2]\) is a \(\LieG\)-module; by \vref{theorem:X.G.G.mod}, its Atiyah class vanishes.
Writing these in the obvious notation, we find
\begin{theorem}
On any complex manifold of complex dimension \(n\) with a holomorphic conformal structure,
\[
(1+2\alpha)c(T)=(1+\alpha)^{n+1},
\]
in Dolbeault cohomology, where \(\alpha=c_1\OO(1)\) is the Chern class in Dolbeault cohomology of the ``ghost line bundle'' \(\OO(1)\).
\end{theorem}
If we don't like ``ghost line bundles'', clearly \(\alpha=2c_1\OO(2)=c_1(T)/(n-1)\).
So we can write this out as
\[
\pr{1+\frac{2c_1(T)}{n-1}}c(T)
=
\pr{1+\frac{c_1(T)}{n-1}}^{n+1}.
\]
Kobayashi and Ochiai expand out this expression explicitly.
Write the dimension \(n\) of our complex manifold as \(n=2m\) or \(n=2m+1\).
For a real variable \(h\), expand in a Taylor series
\[
\sum_{q=0}^m (1+h)^{n-2q}h^{2q}=1+a_1h+a_2h^2+\dots
\]
\begin{theorem}
On any complex manifold of complex dimension \(n\) with a holomorphic conformal geometry,
\[
n^kc_k(T)=a_kc_1(T)^k,
\]
for \(k=1,2,\dots,n\), in Dolbeault cohomology.
\end{theorem}
Kobayashi and Ochiai \cite{Kobayashi/Ochiai:1982} p. 596 theorem 3.20 prove a weaker result, for only the top Chern classes, but in \deRham{} cohomology.
It seems possible to obtain relations on Chern--Simons classes as we did for projective connections.

\section{\texorpdfstring{\Slovak{} cohomology}{Slovak cohomology}}
We have seen that Dolbeault cohomology has the advantage of not feeling curvature or torsion terms in Atiyah classes arising in holomorphic Cartan geometries, but the disadvantage of employing \(C^{\infty}\) reductions of structure group, which are typically not holomorphic, so does not naturally extend into other categories besides complex manifolds.
The holomorphic \deRham{} cohomology has the advantage of being holomorphic, so having natural analogues in other categories, but the curvature and torsion do not drop out of the calculations.
We formulate a cohomology theory, which we call \Slovak{} cohomology, defined in terms of a Cartan geometry, which has both advantages.
\subsection{Models of Cartan geometries}
Take a complex homogeneous space \((X,G)\) with \(x_0\in X\), \(H\defeq G^{x_0}\), and a Langlands decomposition.
Let
\[
\Lambda^{p,q,r}\defeq\Lm{p}{\LieG_-}^*
\otimes\Lm{q}{\LieG_0}^*
\otimes\Lm{r}{\LieG_+}^*
\subset\Lm{p+q+r}{\LieG}^*.
\]
In the splitting, \(G_+\) acts on
\[
\LieG=\LieG_-\oplus \LieG_0\oplus\LieG_+
\]
as
\[
\rho_{\LieG}(g_+)
\begin{pmatrix}
v_-\\
v_0\\
v_+
\end{pmatrix}
=
\begin{pmatrix}
\pi_-g_+&0&0\\
\pi_0g_+&I&0\\
\pi_+g_+&\Ad_{g_+}-I&\Ad_{g_+}
\end{pmatrix}
\begin{pmatrix}
v_-\\
v_0\\
v_+
\end{pmatrix}
\]
so \(\Lambda^{p,q,r}\) might not be an \(H\)-module, only a \(G_0\)-module.

Let
\[
\Lambda^{\ge p,\le r}_{p+q+r}
\defeq
\bigoplus
\Lambda^{p',q',r'}
\]
where the sum is over \(p',q',r'\) with
\begin{align*}
p'&\ge p,\\
r'&\le r,\\
p'+q'+r'&=p+q+r.
\end{align*}
From our matrix above, \(\Lambda^{\ge p,\le r}_{p+q+r}\) is an \(H\)-submodule of \(\Lm{p+q+r}{\LieG}^*\).
Let
\[
\gLm{p}[q]{r}\defeq
\Lambda^{\ge p,\le r}_{p+q+r}/\Lambda^{\ge p,\le r-1}_{p+q+r},
\]
the quotient \(H\)-module.
Take a \(G_0\)-module \(\map[\rho_V]{G_0}{\GL{V}}\).
Extend \(V\) to an \(H\)-module, by making \(G_+\) act trivially.
Let \(\gLm[V]{p}[q]{r}\defeq V\otimes\gLm{p}[q]{r}\).
\subsection{Wedge product}
Wedging \(\Lm{p,q,r}\) gives
\[
\map{%%
\Lambda^{\ge p,\le r}_{p+q+r}
\otimes
\Lambda^{\ge p',\le r'}_{p'+q'+r'}
}{%%
\Lambda^{\ge p+p',\le r+r'}_{p+p'+q+q'+r+r'}%%
}.
\]
This descends to the quotient as the wedge preserves the various numbers of factors.
Hence we have a wedge product
\[
\mapto%%
{%%
\xi\in\gLm{p}[q]{r}, \eta\in\gLm{p'}[q']{r'}
}%%
{%%
\xi\wedge\eta\in\gLm{p+p'}[q+q']{r+r'}%
}%%
.
\]
Similarly, for \(G_0\) modules \(V,W\), we have a wedge product
\[
\mapto%%
{%%
\xi\in\gLm[V]{p}[q]{r}, \eta\in\gLm[W]{p'}[q']{r'}
}%%
{%%
\xi\wedge\eta\in\gLm[V\otimes W]{p+p'}[q+q']{r+r'}%
}%%
.
\]
Clearly a similar but more sophisticated theory arises if we use a more subtle grading of the Lie algebra \(\LieG\), as in the theory of parabolic geometries \cite{Cap.Slovak.2009a}.

\subsection{Cartan geometries}
Take a Cartan geometry \(H \to E \to M\) with model \((X,G)\) with a Langlands decomposition.
Take an \(G_0\)-module \(\map[\rho_V]{G_0}{\GL{V}}\).
Extend \(V\) to an \(H\)-module by making \(G_+\) act trivially.
Let \(W\defeq \gLm[V]{p}[q]{r}\).
Sections of the associated vector bundle are represented by holomorphic \(H\)-equivariant functions \(\map[f]{E}{W}\), i.e. \(r_h^*f=\rho_W(h)^{-1}f\), for any \(h\in H\).

A \emph{\Slovak{} cochain} is an \(G_0\)-equivariant holomorphic map \(\map[f]{E}{W}\).
Since \(f\) is not required to be \(H\)-equivariant, \Slovak{} cochains are not generally sections of the associated vector bundle on \(M\).
Instead they are holomorphic sections of a holomorphic vector bundle over \(E/G_0\).
Nonetheless, we write the set of \Slovak{} cochains as \(\vb{W}=\gForms[V]{p}[q]{r}\).
Locally lift \(f\) to an \(G_0\)-equivariant holomorphic map \(\map[F]{E}{V^{\ge p,\le r}_{p+q+r}}\).
From \(F\), we define a \(V\)-valued form \(\xi\) on \(E\), given by
\[
\xi=Fω-^{p'}\wedgeω0^{q'}\wedgeω+^{r'},
\]
with various sums made implicit, including summing over \(p'\ge p\), \(r'\le r\) with \(p'+q'+r'=p+q+r\).
Let \(\dS\defeq d+ω0\wedge\), so
\[
\dS\xi\defeq d\xi+\rho_V(ω0)\wedge\xi.
\]

As a differential form valued in the vector space \(V\),
\[
\LieDer_v \xi=v\hook d\xi+d(v\hook \xi)
\]
for any vector field \(v\), and in particular for \(v\in\LieH\).
As a \(G_0\)-equivariant \(V\)-valued differential form, \(r_h^* \xi=\rho_V(h)^{-1}\xi\) for any \(h\in G_0\).
Hence for \(v\in\LieG_0\),
\[
\LieDer_v \xi = -\rho_V(v)\xi,
\]
\begin{lemma}\label{lemma:v.hook}
On \Slovak{} cochains, \(0=v\hook\dS+\dS v\hook\) for \(v\in\LieG_0\).
\end{lemma}
\begin{proof}
If we let \(\eta=v\hook\xi\), we find
\begin{align*}
v\hook\dS\xi
&=
v\hook(d\xi+\rho_V(ω0)\wedge\xi),
\\
&=
-d(v\hook\xi)+\LieDer_v\xi+\rho_V(v)\xi-\rho(ω0)(v\hook\xi),
\\
&=
-d\eta-\rho_V(v)\xi+\rho_V(v)\xi-\rho(ω0)\eta,
\\
&=
-\dS\eta.
\end{align*}
\end{proof}
Note that
\begin{align*}
\dSω-&=dω- + \rho_{\LieG_-}(ω0)\wedgeω-,\\
\dSω0&=dω0 + \frac{1}{2}\lb{ω0}{ω0},\\
\dSω+&=dω+ + \rho_{\LieG_+}(ω0)\wedgeω+.
\end{align*}
In particular, \(\dS\) has components which can raise indices \((p,q,r)\) by
\begin{align*}
\left.
\begin{matrix}
(1,0,0)
\end{matrix}
\right\}&\text{ if  }p>0,\\
\left.
\begin{matrix}
(1,-1,1)\\
(2,-1,0)
\end{matrix}
\right\}&\text{ if }q>0\\
\left.
\begin{matrix}
(2,0,-1)\\
(0,0,1)
\end{matrix}
\right\}&\text{ if }r>0\\
\left.
\begin{matrix}
(1,0,0)\\
(0,0,1)
\end{matrix}
\right\}&\text{ for \(\dS{F}\)}.
\end{align*}
So \(\dS\) takes
\[
\map{%%
\vb{V}^{p,q,r}
}%%
{%%
\vb{V}^{p+2,q-1,r}\oplus
\vb{V}^{p+2,q,r-1}\oplus
\vb{V}^{p+1,q,r}\oplus
\vb{V}^{p+1,q-1,r+1}\oplus
\vb{V}^{p,q,r+1}%
}%%
\]
and takes
\[
\map{%%
\vb{V}^{\ge p,\le r}_{p+q+r}
}%%
{%%
\vb{V}^{\ge p,\le r+1}_{p+q+r+1}%
}%%
,
\]
so descends to the quotient, the \emph{\Slovak{} differential}
\[
\map[d]{\gForms[V]{p}[q]{r}}{\gForms[V]{p}[q]{r+1}}.
\]
In the quotient, \(\rho_V(ω0)\wedge=0\), so we can write \(\dS\) as \(d\).
The \Slovak{} differential is the quotient of the exterior derivative, because the connection form term drops out, but the Cartan geometry identifies sections with forms so the cohomology depends on the Cartan geometry.
Hence \(d^2=0\), i.e. the \Slovak{} differential is a differential, with cohomology the \emph{\Slovak{}} cohomology
\[
\cohomology{p,q,r}{M,\vb{V}}.
\]

When we compute in the quotient, all curvature terms lie in the denominator so
\begin{align*}
\dSω-&=-\rho_{\LieG_-}(ω+)\wedgeω-,\\
\dSω0&=a(ω+,ω-),\\
\dSω+&=-\frac{1}{2}\lb{ω+}{ω+}.
\end{align*}
If we write out a cochain \(\xi\) as
\[
\xi=Fω-^{p'}\wedgeω0^{q'}\wedgeω+^{r'}
\]
then we can write
\[
\dS F = F^-ω- + F^0ω0 + F^+ω+.
\]
In the quotient we fully expand out
\[
\dS\xi
=
F^+ω+ ʌω-^{p'} ʌω0^{q'} ʌω+^r
+
F(\dSω- ʌ\dots).
\]
The differential equation of \Slovak{} closure in the quotient, as equations on \(F\), form a constant coefficient linear equation, depending only on the infinitesimal model and on \(F^+\) and \(F\) at each point.

\begin{example}
For projective connections, the action of \(h\in G_+=\C[n*]\) on \(ω0=(\omega^i_j)\) is
\[
r_h^*\omega^i_j=\omega^i_j+(\delta^i_j h_k+\delta^i_k h_j)\omega^k.
\]
\end{example}
\begin{lemma}
If a \Slovak{} cycle is \(G_+\)-invariant, then it is closed.
\end{lemma}
\begin{proof}
Repeating the argument from~\vref{lemma:v.hook}, if \(\xi\) is an \(G_+\)-invariant cochain, then \(0=v\hook\dS\xi+\dS(v\hook\xi)\) for any \(v\in\LieG_+\).

Take two vectors \(v_1,v_2,\in\LieH_+\).
Note that \(v_1\hook\xi\) is not \(H_+\)-invariant, but the ``error''
\[
\LieDer_{v_2}(v_1\hook\xi)=\lb{v_2}{v_1}\hook\xi,
\]
is lower order, i.e. has \(r-1\) copies of \(ω+\).
Similarly when we repeatedly hook, we don't get invariant differential forms, but the ``error'' vanishes in the quotient.
Hooking in \(r+1\) vectors \(v_1,\dots,v_{r+1}\in\LieH_+\),
\[
(v_1,\dots,v_{r+1})\hook\dS\xi
=
\pm\dS((v_1,\dots,v_{r+1})\hook\xi)=0
\]
in the quotient.
So in the quotient, \(\dS\xi=0\).
\end{proof}
\begin{lemma}\label{lemma:what.is.Cech}
On \Slovak{} cochains \(\xi\) in \(\gForms[V]{p}[q]{0}\), i.e. \(p,q,r\) with \(r=0\), \(\dS\xi=0\) just when \(\xi\) is \(G_+\)-invariant and hence the pullback of a holomorphic section of \(\vb{V}^{p,q,0}\defeq V\otimes \Omega^p_M \otimes \ad{\qE}^{\otimes q}\) from \(\qE=E/G_+\).
\end{lemma}
\begin{proof}
The denominator in the definition of \(\gForms[V]{p}[q]{0}\) vanishes, i.e. its sections are differential forms.
Clearly the following are equivalent:
\begin{itemize}
\item
The differential \(\nabla\xi\) vanishes in \(\gForms[V]{p}[q]{1}\) 
\item
\(d\xi+\omega_0\wedge\xi\) belongs to \(\vb{V}^{\ge p,\le 0}_{p+q+1}\)
\item
\(0=v\hook(d\xi+\omega_0\wedge\xi)\) for any \(v\in\LieG_+\)
\item
\(0=v\hook d\xi\) for any \(v\in\LieG_+\) (since \(v\hook\omega_0=0\) and \(v\hook\xi=0\) already),
\item
\(0=v\hook d\xi+d(v\hook\xi)\) for any \(v\in\LieG_+\) (since \(v\hook\xi=0\) already),
\item
\(0=\LieDer_v \xi\) for any \(v\in\LieG_+\),
\item
\(0=r_{e^{tv}}^* \xi\) for any \(v\in\LieG_+\),
\item
\(0=r_{g_+}^* \xi\) for any \(g_+\in G_+\) (since \(G_+\) is connected).
\end{itemize}
\end{proof}
We have already argued that \(\dS^2=0\), but to be clear we check again:
\begin{lemma}
The \Slovak{} differential \(\dS\) is a differential, i.e. \(\dS^2=0\).
\end{lemma}
\begin{proof}
Since
\begin{align*}
\dS^2
&=
(d+ω0\wedge)(d+ω0\wedge),
\\
&=
d^2+d(ω0\wedge)+ω0ʌd + ω0^2,
\\
&=
0+(dω0)\wedge-(ω0ʌd)+ω0ʌd + ω0^2,
\\
&=
(dω0+ω0^2)\wedge,
\\
&=
\rho_V(\dSω0)\wedge,
\\
&
=a_V(ω-,ω+)\wedge,
\end{align*}
is \((1,0,1)\), we see that on the quotient \(\dS^2=0\).
\end{proof}
\begin{lemma}\label{lemma:right.H.plus}
The \Slovak{} differential is \(H\)-invariant.
So \Slovak{} cohomology is an \(H\)-module.
\end{lemma}
\begin{proof}
Under right \(H_+\)-action on \(E\), on \(\vb{V}^{p,q,r}\),
\begin{align*}
r_{h_+}^*\dS
&=
r_{h_+}^*(d+\rho_V(ω0)\wedge),
\\
&=
dr_{h_+}^*
+
\rho_V(r_{h_+}^*ω0)ʌr_{h_+}^*,
\\
&=
dr_{h_+}^*
+
\rho_V(ω0-\pi_0h_+ω-)ʌr_{h_+}^*,
\\
&=
\dS r_{h_+}^*
-
\rho_V(\pi_0h_+ω-)ʌr_{h_+}^*,
\end{align*}
has ``error term'' \(\rho_V(\pi_0h_+ω-)\wedge\) of type \((1,0,0)\), so vanishing in the quotient, i.e. \(\dS\) becomes \(H\)-invariant in the quotient.
\end{proof}
Since the fibers of \(\map{E/G_0}{M}\) are contractible and Stein, the sheaf cohomology of any holomorphic vector bundle over \(M\) pulls back by isomorphism when we pull back vector bundles to \(E/G_0\) \cite{Buchdahl:1983} p. 365.
\begin{lemma}
On any Stein open set in \(M\), the \Slovak{} cohomology is trivial in positive degree.
It is therefore isomorphic to \Cech{} cohomology of the vector bundle \(\vb{V}^{p,q,0}\defeq  V\otimes \Omega_M \otimes \ad{\qE}\) over the image in \(M\) of that open set.
\end{lemma}
\begin{proof}
We can assume that \(M\) is Stein and hence \(E/G_0\) is Stein.
By a theorem of Serre \cite{Matsushima/Morimoto:1960} p.~142 Proposition 5, the preimage of \(E/G_0\) in \(E\) is Stein.
We have seen that the \Slovak{} differential in the quotient \(\gForms[V]{p}[q]{r}\) is the exterior derivative.
So on any Stein open set in \(E\), any \Slovak{}-closed form of positive degree \(r\) is \(d\)-exact.
Since the group \(G_0\) is reductive, it contains a maximal compact real subgroup which is Zariski dense.
Averaging over the Haar measure of this subgroup, we ensure that any \Slovak{}-closed form is \(G_0\)-invariantly exact.
\end{proof}

\subsection{Characteristic classes}
We define as above the Chern classes, Chern--Simons classes, Atiyah class, and so on, in \Slovak{} cohomology.
Note that the Chern--Simons forms vanish as \Slovak{} classes, except perhaps for the final term
\[
\CS{f}=f(ω0,dω0^{k-1})
\]
if \(f\) is a homogeneous polynomial on \(\LieG_0\) of degree \(k\), because other terms have lower grade \(r\) in the trigrading \(p,q,r\).
Therefore these classes are much easier to compute in \Slovak{} cohomology than Dolbeault Chern--Simons classes of holomorphic principal or vector bundles more generally.
\begin{remark}
All of the theorems we have stated above about Chern classes and Chern--Simons classes in Dolbeault cohomology have the same proofs in \Slovak{} cohomology.
\end{remark}
\begin{example}
If a holomorphic vector bundle is the pull back of the universal bundle from a holomorphic map to the Grassmannian, or more generally to some complex manifold with a holomorphic Grassmannian geometry, we can use the above apparatus to define a kind of \Slovak{} cohomology for that bundle, which maps then to the Dolbeault cohomology, taking the pulled back \Slovak{} Chern classes to the usual Chern classes in Dolbeault cohomology.
\end{example}

\section{\texorpdfstring{\Slovak{} to Dolbeault}{Slovak to Dolbeault}}%
\label{section:CdBshapes}
\subsection{The Dolbeault map}
Take a complex homgeneous space \((X,G)\) with point \(x_0\in X\), let \(H\defeq G^{x_0}\), take a Langlands decomposition \(H=G_0\ltimes G_+\) and take a holomorphic \((X,G)\)-Cartan geometry \(H \to E \to M\).
Recall we let \(\qE\defeq E/G_+\), a holomorphic principal right \(G_0\)-bundle \(\bundle*{G_0}{\qE}{M}\).
Recall that since \(H/G_0\) is contractible, \(\map{E/G_0}{M}\) admits a \smooth section \(\map[s]{M}{E/G_0}\) i.e. a \smooth \(G_0\)-reduction of structure group.
The pullback
\[
\begin{tikzcd}
s^*E \arrow{r} \arrow{d} & E \arrow{d} \\
M \arrow{r} & E/G_0
\end{tikzcd}
\]
makes \(s^*E\) a \(G_0\)-bundle equivariantly mapped to \(E\), and composing \(s^* E \to E \to \qE\) gives an isomorphism of \(G_0\)-bundles.
The inverse isomorphism is a \(G_0\)-equivariant map which we also denote by \(\map[s]{\qE}{E}\).
Conversely, every \(G_0\)-equivariant map \(\map[s]{\qE}{E}\) quotients to a reduction \(\map[s]{M}{E/G_0}\).

The pullback by this map \(\map[s]{\qE}{E}\) takes \(ω-\) to a semibasic \((1,0)\)-form \(s^*ω-\), since \(ω-\) is a semibasic \((1,0)\)-form on \(E\).
Since \(s^*ω0\) is a \((1,0)\)-connection, it is in particular a \((1,0)\)-form.
However, the pullback \(s^*ω+\) may have \((1,0)\) and \((0,1)\) parts.
Fix a \(G_0\)-module \(V\).
Recall that \(\vb{V}^{p,q,r}\) is the collection of holomorphic differential forms \(\xi\) on \(E\) valued in the associated vector bundle \(\vb{V}\), which have type \((p,q,r)\) in the Langlands decomposition \(\LieG=\LieG_-\oplus\LieG_0\oplus\LieG_+\).
So the pullback \(s^*\xi\) of \(\xi\) to \(s^*E\) has Dolbeault decomposition into parts
\[
(p+q+r,0), (p+q+r-1,1), \dots, (p+q,r).
\]
To each such \(\xi\), we associate the \(\vb{V}\)-valued \((p+q,r)\)-form \(\Dlb{\xi}=(s^*\xi)^{p+q,r}\) defined as the \((p+q,r)\)-part of the pullback to \(\qE\):
\[
\mapto{\vb{V}^{p,q,r}}{\vb{V}^{p+q,r}}.
\]
This map takes
\[
\map{%%
\vb{V}^{\ge p,\le r}_{p+q+r}
}%%
{%%
\vb{V}^{p+q,r}%
}%%
.
\]
Hence it takes
\[
\map{%%
\vb{V}^{\ge p,\le r-1}_{p+q+r}
}%%
{%%
0%
}%%
.
\]
So it is defined on the quotient.
The resulting object \(\Dlb{\xi}\) is a \smooth{} section of
\[
\vb{V}\otimes\Omega^{p,r}\otimes (\ad \qE)^{*\otimes q}.
\]
The Dolbeault map preserves wedge products.
The \Slovak{} operator is taken to
\begin{align*}
\Dlb*{\dS\xi}
&=
(s^*\dS\xi)^{p+q,r+1},
\\
&=
(s^*d\xi + \rho_V(s^*ω0)ʌs^*\xi)^{p+q,r+1},
\\
&=
(ds^*\xi)^{p+q,r+1},
\\
&=
((\partial+\bar\partial) s^*\xi)^{p+q,r+1},
\\
&=
\bar\partial (s^*\xi)^{p+q,r},
\\
&=
\bar\partial\,\Dlb{\xi}.
\end{align*}
Hence the \Slovak{} operator becomes the Dolbeault operator, and we map \Slovak{} cohomology to Dolbeault cohomology.
The Dolbeault isomorphism then maps \Slovak{} cohomology to sheaf cohomology
\[
\cohomology{r}{M,\vb{V}\otimes\Omega^p\otimes (\ad \qE)^{*\otimes q}}.
\]
In particular, when \(0=p=q\) we compute in sheaf cohomology
\(
\cohomology{r}{M,\vb{V}}.
\)

\subsection{Changing reduction}
Take two \(C^{\infty}\) reductions of structure group, i.e. \(C^{\infty}\) sections \(\map[s_0,s_1]{M}{E/G_0}\).
They are equivalent to \(G_0\)-equivariant maps \(\map[s_0,s_1]{\qE}{E}\), so \(s_1=s_0g_+\) for a unique \(C^{\infty}\) \(\map[g_+]{\qE}{G_+}\) with
\[
r^*_{g_0}g_+=\Ad_{g_0}^{-1} g_+
\]
for \(g_0\in G_0\).
\begin{lemma}\label{lemma:H.plus.action}
\(s_1^*\omega
=
\Ad_{g_+}^{-1} s_0^*  \omega
+
g_+^{-1} dg_+.\)
\end{lemma}
\begin{proof}
Pick a point \(\qe_0\in \qE\), a vector \(v \in T_{\qe_0} \qE\) and a curve \(\qe(t)\) with \(\qe(0)=\qe_0\) and \(\qe'(0)=v\).
Let
\[
u\defeq\left.\frac{d}{dt}\right|_{t=0} g_+(\qe_0)^{-1}g_+(\qe(t))
= \pr{g_+(\qe_0)^{-1}}'_{g_+(\qe_0)}g'_+(\qe_0)v.
\]
Compute
\begin{align*}
s_1'(\qe_0)v
&=
\left.\frac{d}{dt}\right|_{t=0}
s_1(\qe(t)),
\\
&=
\left.\frac{d}{dt}\right|_{t=0}
s_0(\qe(t))g_+(\qe(t)),
\\
&=
\left.\frac{d}{dt}\right|_{t=0}
s_0(\qe(t))g_+(\qe_0)
+
\left.\frac{d}{dt}\right|_{t=0}
s_0(\qe_0)g_+(\qe(t)),
\\
&=
r'_{g_+(\qe_0)}(s_0(\qe_0))s_0'(\qe_0)v
+
u(s_1(\qe_0)).
\end{align*}
Hence
\begin{align*}
s_1^*\omega(v)
&=
\omega(s_1'(\qe_0)v),
\\
&=
\omega(r'_{g_+(\qe_0)}(s_0(\qe_0))s_0'(\qe_0)v
+
u(s_1(\qe_0))),
\\
&=
(s_0^* r_{g_+(\qe_0)}^* \omega)(v)
+u
\\
&=
(s_0^* r_{g_+(\qe_0)}^* \omega)(v)
+
g_+^{-1} dg_+(v),
\\
&=
(s_0^* \Ad_{g_+(\qe_0)}^{-1} \omega)(v)
+
g_+^{-1} dg_+(v),
\\
&=
(\Ad_{g_+(\qe_0)}^{-1} s_0^* \omega)(v)
+
g_+^{-1} dg_+(v).
\end{align*}
\end{proof}
Expressed in components of the Cartan connection,
\[
s_1^*
\begin{pmatrix}
ω-\\
ω0\\
ω+
\end{pmatrix}
=
\rho(g_+)^{-1}
s_0^*
\begin{pmatrix}
ω-\\
ω0\\
ω+
\end{pmatrix}
+
\begin{pmatrix}
0\\
0\\
g_+^{-1}dg_+
\end{pmatrix}
\]
where
\[
\rho(g_+)^{-1}
=
\begin{pmatrix}
\pi_-g_+^{-1}&0&0\\
-\pi_0g_+&I&0\\
(I-\Ad_{g_+}^{-1})\pi_0g_++\Ad_{g_+}^{-1}\pi_+g_+&\Ad_{g_+}^{-1}-I&\Ad_{g_+}^{-1}
\end{pmatrix}.
\]

\section{\texorpdfstring{\Cech--Dolbeault cohomology}{Cech Dolbeault cohomology}} %
\label{section:CdB}
We briefly recall \Cech--Dolbeault cohomology \cite{Abate2013}, since we need the associated notation below.
Fix a complex manifold \(M\), a holomorphic vector bundle \(\map{\vb{V}}{M}\), and a cover \(\set{M_a}\) by open sets \(M_a \subset M\).
To each tuple \(A\defeq\pr{a_0,a_1,\dots,a_k}\) assign
\[
M_A \defeq M_{a_0} \cap M_{a_1} \cap \dots \cap M_{a_k}
\]
and let
\[
\drop{i}{A} \defeq \pr{a_0,a_1,\dots,\hat{a}_i,\dots,a_k}.
\]

Let \(M_k\) be the disjoint union of all choices of \(k\)-fold overlaps:
\[
M_k \defeq \bigsqcup M_A,
\]
for \(A=\pr{a_0,a_1,\dots,a_k}\).
Each point of \(M_k\) is a pair \((x,A)\), where \(x \in M_A\).
Define local biholomorphisms
\[
\mapto[\inci]{(x,A) \in M_k}{(x,\drop{i}{A}) \in M_{k-1}},
\]
for \(i=0,1,\dots,k\):
\[
\tikzcdset{arrow style=tikz,diagrams={>={Computer Modern Rightarrow[scale=0.8,width=3pt]}}}
\begin{tikzcd}
\cdots\arrow[r]\arrow[r,shift left=1]\arrow[r,shift right=1]\arrow[r,shift left=2]\arrow[r,shift right=2]&
M_2\arrow[r]\arrow[r,shift left=1.5]\arrow[r,shift right=1.5]&
M_1\arrow[r,shift left]\arrow[r,shift right]&
M_0\arrow[r]&
M
\end{tikzcd}
\]
The \Cech{} differential \(\delta\) is
\[
\delta \xi= \sum_{i=0}^k (-1)^i \inci^* \xi
\]
for \(\xi\) on \(M_k\) a \smooth \(\vb{V}\)-valued differential form.

Let \(\cpx{M}\) be the disjoint union of the \(M_k\).
A \smooth \(\vb{V}\)-valued \emph{\((p,q)\)-form} on \(\cpx{M}\) is a \smooth \(\vb{V}\)-valued \((p,q-k)\) form on each \(M_k\), changing signs when we permute the entries in any \(A=(a_0,\dots,a_k)\).
Let \((-1)^{\bullet}\) be the function \((-1)^k\) on \(M_k\).
The \Cech--Dolbeault differential is \(\CdB=\delta + (-1)^{\bullet} \bar\partial\).
The \Cech--Dolbeault \((p,q)\)-cochains of \(\vb{V}\) for the given open cover are the \smooth \(\vb{V}\)-valued \((p,q)\)-forms on \(\cpx{M}\).
For example
\[
\CdB \set{\xi_a,\xi_{ab},\dots} \defeq \set{\bar\partial\xi_a,\xi_b-\xi_a-\bar\partial\xi_{ab},\dots}.
\]
The \Cech--Dolbeault differential takes \((p,q)\)-cochains to \((p,q+1)\)-cochains, and satisfies \(\CdB^2=0\).
For any Stein cover, the usual Dolbeault cohomology is isomorphic to \Cech--Dolbeault cohomology by \(\mapto{\xi}{\set{\xi_a,\xi_{ab},\dots}\defeq\set{\xi,0,0,\dots}}\), see \cite{Abate2013}.

Take a complex Lie group \(H\) and a holomorphic principal right \(H\)-bundle \(H \to E \to M\).
Suppose that each \(E_a \defeq \left.E\right|_{M_a}\) has a \smooth complex linear connection \(1\)-form \(ωa\).
The \emph{Atiyah class} of \(\map{E}{M}\) is the \Cech--Dolbeault class of
\[
a(M,E)=\set{\xi_a,\xi_{ab},\xi_{abc},\dots}
=
\set{\bar\partialωa,ωa-ωb,0,0,\dots}=\CdB\set{ωa,0,0,\dots}.
\]
Careful: \(a(M,E)\) is written here as an \(\ad{E}\)-valued \((1,1)\)-chain on \(M\), and as \(\CdB\) of a cochain, but not as \(\CdB\) of an \(\ad{E}\)-valued \((1,0)\)-chain on \(M\); only of such a cochain on \(E\).
Nonetheless, this makes clear that \(a(M,E)\) is \(\CdB\)-closed on \(M\), as it is \(\CdB\)-exact on \(E\), and \(\CdB\) commutes with pullback.
This class is independent of the choice of connections and vanishes if and only if \(\map{E}{M}\) admits a holomorphic connection \cite{Abate2013}.

\section{Czechoslovak cohomology}
We develop a \Cech{} variant of \Slovak{} cohomology.
Fix a holomorphic principal right \(H\)-bundle \(\map{E}{M}\) and a cover \(\set{M_a}\) of \(M\) by open sets \(M_a \subset M\).
Adopt notation as for \Cech--Dolbeault cohomology.
Let \(E_A\defeq \left.E\right|_{M_A}\).
Pick a complex homogeneous space \((X,G)\) with Langlands decomposition.
On \(\map{E}{M}\), pick a holomorphic \((X,G)\)-geometry with Cartan connection \(\omega\).
A holomorphic section of \(\gForms[V]{p}[q]{r}\) on \(\cpx{M}\) is a holomorphic section of \(\gForms[V]{p}[q]{r-k}\) on each \(M_k\) which changes sign when two indices are permuted in a tuple \(A=(a_0,\dots,a_k)\):
\[
\gForms[V]{p}[q]{r}_{\cpx{M}}
\subset
\gForms[V]{p}[q]{r}_{M_0}
\oplus
\gForms[V]{p}[q]{r-1}_{M_1}
\dots
\oplus
\gForms[V]{p}[q]{0}_{M_r}.
\]
Let \(F\) be the function \(F=k\) on \(M_k\).
The \emph{Czechoslovak differential} is \(\CSl=\delta + (-1)^{\bullet}\dS\).
The \emph{Czechoslovak} \((p,q,r)\)-cochains of \(\vb{V}\) for the given open cover are holomorphic sections of \(\gForms[V]{p}[q]{r}\) on \(\cpx{M}\).
For example
\[
\CSl \set{\xi_a,\xi_{ab},\dots} \defeq \set{\dS\xi_a,\xi_b-\xi_a-\dS\xi_{ab},\xi_{bc}+\xi_{ca}+\xi_{ab}+\dS\xi_{abc},\dots}.
\]
\begin{lemma}
The Czechoslovak differential takes \((p,q,r)\)-cochains to \((p,q,r+1)\)-cochains, and satisfies \(\CSl^2=0\).
\end{lemma}
\begin{proof}
Writing \(A=(a_0,\dots ,a_k)\), take a holomorphic local section \(\xi\) of \(\gForms[V]{p}[q]{r-k}\) on each \(M_A\), say represented by a local section of \(\vb{V}^{p,q,r-k}\).
So
\begin{align*}
(\delta \xi)_{(x,A)}
&=
\sum_{i=0}^k(-1)^i (\inci^* \xi)_{(x,A)},
\\
&=
\sum_{i=0}^k(-1)^i \xi_{(x,\drop{i}{A})},
\end{align*}
increasing value of \(r\) by \(1\).
We note that \((-1)^{\bullet}\delta=-\delta(-1)^{\bullet}\), as \(\delta\) decreases the number of indices.
\begin{align*}
(\CSl^2\xi)_{(x,A)}
&=
((\delta+(-1)^{\bullet}\dS)^2\xi)_{(x,A)},
\\
&=
(\delta^2\xi)_{(x,A)}
+
(\delta(-1)^{\bullet}\dS\xi)_{(x,A)}
+
(\dS(-1)^{\bullet}\delta\xi)_{(x,A)}
+
(\dS^2\xi)_{(x,A)},
\\
&=
(\delta(-1)^{\bullet}\dS\xi)_{(x,A)}
-
((-1)^{\bullet}\dS\delta\xi)_{(x,A)}.
\end{align*}
\end{proof}
The \emph{Czechoslovak cohomology}
\[
\cohomology{p,q,r}{\cpx{M},\vb{V}}
\]
is the cohomology of the Czechoslovak differential.
Note that it will in general depend on the choice of open cover \(\cpx{M}\) of \(M\).
The \emph{\u{C}ech map} is
\[
\mapto{\xi \in \gForms[V]{p}[q]{r}}{\set{\xi_a, \xi_{ab}, \dots}
\defeq
\set{\xi,0,0,\dots}}.
\]
\begin{theorem}
The \Cech{} map is a cochain map from \Slovak{} cochains to Czechoslovak cochains, injective in cohomology.
\end{theorem}
\begin{proof}
Write the \u{C}ech map as \(\mapto{\xi}{\Cch{\xi}}\).
So
\begin{align*}
\Cch{\xi}_a &= \xi, \\
\Cch{\xi}_{ab} &= 0, \\
\Cch{\xi}_{abc} &= 0, \\
&\vdots
\end{align*}
Using \(\dS\xi\) in place of \(\xi\),
\begin{align*}
\Cch*{\dS\xi}_{a} &= \dS\xi, \\
\Cch*{\dS\xi}_{ab} &= 0, \\
\Cch*{\dS\xi}_{abc} &= 0, \\
&\vdots
\end{align*}
So
\begin{align*}
(\CSl\Cch{\xi})_a &= \dS\Cch{\xi}_a = \dS\xi \\
(\CSl\Cch{\xi})_{ab} &= \Cch{\xi}_b-\Cch{\xi}_a-\dS\Cch{\xi}_{ab} = 0, \\
(\CSl\Cch{\xi})_{abc} &= 0+(-1)^2 0=0, \\
&\vdots
\end{align*}
using the hypothesis that \(\dS\xi=0\) and \(\xi_a=\xi_b=\xi\) is globally defined.
So \(\mapto{\xi}{\Cch{\xi}}\) is a cochain map.

Suppose that \(\Cch{\xi}=\CSl\eta\).
Then
\begin{align*}
\Cch{\xi}_a &= \xi = \dS\eta_a, \\
\Cch{\xi}_{ab} &= 0 = \eta_b-\eta_a-\dS\eta_{ab}, \\
&\vdots
\end{align*}
So \(0=\dS\eta_b-\dS\eta_a-\dS^2\eta_{ab}\), i.e. \(\dS\eta_a\) is globally defined, and \(\xi=\dS\eta_a\).
Therefore the \Cech{} map is injective in cohomology.
\end{proof}

\subsection{Cup product}
The \emph{cup product} of Czechoslovak classes \(\xi,\eta\) is
\[
(\xi\cup\eta)_{a_0\dots a_k}
\defeq
\sum_{j=0}^k
(-1)^{(p+q+r-j)(k-j)}
\xi_{a_0\dots a_j}\wedge\eta_{a_j\dots a_k}.
\]
if \(\xi\) is a section of \(\gForms[V]{p}[q]{r}\) and \(\eta\) a section of \(\gForms[W]{p}[q]{r}\), so \(\xi\cup\eta\) is a section of \(\gForms[V\otimes W]{p+p'}[q+q']{r+r'}\).
Note that \(\CSl(\xi\cup\eta)=(\CSl\xi)\cup\eta+(-1)^{p+q+r}\xi\cup\CSl\eta\) as in \cite{Suwa:2000}.

\subsection{Quotient cochains}
A \emph{quotient cochain} is a sequence
\[
\set{\xi_{ab}, \eta_{ab}, \eta_{abc}, \dots}
\in
\gForms[V]{p}[q]{r}_{M_1}
\oplus
\gForms[V]{p}[q]{r-1}_{M_1}
\dots
\oplus
\gForms[V]{p}[q]{0}_{M_r}
\]
with \(\xi_{ab}\in\gForms[V]{p}[q]{r}\), \(\eta_{ab}\in\gForms[V]{p}[q]{r-1}\) both on \(M_{ab}\), but after the first step, \(\eta_A\) is just a usual Czechoslovak cochain.
Map each Czechoslovak cochain
\[
\set{\xi_a, \xi_{ab}, \dots}
\]
to a quotient cochain by
\[
\mapto%%
{%%
\set{\xi_a, \xi_{ab}, \dots}
}%%
{%%
\set{\xi_a-\xi_b, \xi_{ab}, \dots}%
}%%
.
\]
The kernel consists precisely of the image of the \Slovak{} cochains in the Czechoslovak cochains.
On the quotient cochains, we use the same Czechoslovak differential:
\[
\CSl\set{\xi_{ab}, \eta_{ab}, \eta_{abc}, \dots}
=
\set{\dS\xi_{ab}, \xi_{ab}-\dS\eta_{ab}, \eta_{bc}+\eta_{ca}+\eta_{ab}+\dS\eta_{abc},\dots}.
\]
The map to quotient cochains makes a short exact sequence of complexes: \Slovak{} cochains to Czechoslovak cochains to quotient chains.
\begin{lemma}
Suppose that all triple overlaps of sets \(M_a\) in our cover are empty.
Then the quotient cohomology is trivial, so \Slovak{} cohomology is isomorphic to Czechoslovak cohomology.
\end{lemma}
\begin{proof}
Take a quotient cocycle
\[
\xi=\set{\xi_{ab}, \eta_{ab}}.
\]
The condition that \(\xi\) is a cocycle is
\[
0=\CSl\xi=\set{\dS\xi_{ab}, \xi_{ab}-\dS\eta_{ab}}.
\]
So \(\xi_{ab}=\dS\eta_{ab}\).
Let
\[
\zeta=\set{\eta_{ab},0},
\]
so
\[
\CSl\zeta=\set{\dS\eta_{ab},\eta_{ab}-\dS0}=\xi.
\]
\end{proof}
\begin{lemma}\label{lemma:zig.zag}
For any open cover of \(M\), the \Cech{} cohomology of \(\vb{V}^{p,q,0}\defeq \vb{V}\otimes\nForms{p}{M}\otimes\Lm{q}{\ad\qE}\) of that open cover injects in the Czechoslovak cohomology of \(\gForms[V]{p}[q]{*}\) of the same open cover.
\end{lemma}
\begin{proof}
By \vref{lemma:what.is.Cech}, on any open set, each \Slovak{} cocycle is identified precisely with a holomorphic section of \(\vb{V}^{p,q,0}\).
Take any \(\eta\) in the \Cech{} cohomology of \(\vb{V}^{p,q,0}\) on \(M\) and identify it with a \Slovak{}-closed form in \(\gForms[V]{p}[q]{0}\) with the same name \(\eta\).
Map it to
\[
\eta'=\set{0,0,\dots,0,\eta_A}.
\]
The differential commutes (up to an irrelevant sign).
So each \(\delta\)-exact \(\eta\) maps to some \(\eta'\) both \(\delta\) and \(\dS\) exact so \(\CSl\)-exact.
Conversely, \(\eta'\) is Czechoslovak exact just when \(\eta=\delta \alpha + \dS\beta\) for some \(\alpha,\beta\).
But \(\eta\) is in \(\gForms[V]{p}[q]{0}\) so \(\beta=0\) and so \(\eta\) is \(\delta\)-exact.
\end{proof}
\begin{remark}
On \(E/G_0\), let 
\begin{align*}
\Omega_-&\defeq(\omega_-,\omega_+), \\
\Omega_0&\defeq\omega_0.
\end{align*}
In this way, the \((X,G)\)-geometry on \(M\) induces an \((X',G')\)-geometry on \(M'\defeq E/G_0\), where \(G'\defeq (\LieG_-\oplus\LieG_+)\ltimes G_0\) and \(H'\defeq G_0\) so \(X'=G'/H'\) with Cartan connection \(\Omega\defeq\Omega_-\oplus\Omega_0\) and Langlands decomposition \(G'_0\defeq G_0\).
By the same argument, for an open covering \(\set{M_a}\) of \(M\), with corresponding preimages \(\set{M'_a}\) in \(M'\), the Czechoslovak cohomology of the \((X',G')\)-geometry for that open covering is also the \Cech{} cohomology of \(\vb{V}^{p,q,0}\) of that covering on \(M'\) or equivalently of the original covering on \(M\).
\end{remark}
\begin{proposition}\label{proposition:CS.to.C}
The Czechoslovak cohomology of \(\gForms[V]{p}[q]{*}\) over any Stein open cover is isomorphic to the \Cech{} cohomology of \(\map{\vb{V}^{p,q,0}}{M}\).
\end{proposition}
\begin{proof}
We follow \cite{Bott/Tu:1982} p. 98 theorem 8.9: take any \(r\)-cocycle \(\xi\).
It lies in our double complex along a diagonal line, with finitely many terms.
Each term \(\xi_k\) maps to the right and up, and the sums or differences vanish, with suitable signs.
The term on the upper left maps to zero by \(\dS\), and that on the lower right maps to zero by \(\delta\).
On a Stein open cover, \(\dS\)-closed implies \(\dS\)-exact on each open set, so the upper left is \(\dS\)-exact: \(\xi_k=\dS\eta_k\).
Replace \(\xi_{k+1}\) by \(\xi_{k+1}\pm \delta(\eta_k)\), we can get this to be \(\dS\)-exact, and so on.
Finally, we leave only \(\xi_r\), i.e. the part arising from \Cech{} cohomology.
\end{proof}
\begin{lemma}\label{lemma:Stein.H.zero}
If \(E/G_0\) is a Stein manifold then the \Slovak{} cohomology group of the cochains \(\gForms[V]{p}[q]{r}\) is isomorphic to Czechoslovak cohomology and hence to the sheaf cohomology group
\[
\cohomology{r}{M,\vb{V}^{p,q,0}}
\]
where \(\map{\vb{V}^{p,q,0}\defeq\vb{V}\otimes\nForms{p}{M}\otimes\Lm{q}{\ad\qE}}{M}\).
\end{lemma}
\begin{proof}
Since the fibers of \(\map{E/G_0}{M}\) are contractible and Stein, the sheaf cohomology of any holomorphic vector bundle over \(M\) pulls back by isomorphism when we pull back vector bundles to \(E/G_0\) \cite{Buchdahl:1983} p. 365.
Just use the open cover consisting of one open set.
\end{proof}

\subsection{Czechoslovak characteristic classes}
The Czechoslovak \emph{Atiyah class} of \(\map{E}{M}\) is the Czechoslovak class image of the \Slovak{} Atiyah class by the \Cech{} map:
\[
a(M,E)
=
\set{\xi_a,\xi_{ab},\xi_{abc},\dots}
=
\set{\dS ω0,0,0,\dots}.
\]

\subsection{\texorpdfstring{Czechoslovak to \Cech{}--Dolbeault}{Czechoslovak to Cech-Dolbeault}}
Take a \smooth{} reduction of structure group \(\map[s_a]{M_a}{E_a/G_0}\).
The \emph{Dolbeault map} on Czechoslovak cohomology is
\[
\Dlb{\set{\xi_a,\xi_{ab},\dots}}
=
\set{\Dlb{\xi_a},\Dlb{\xi_{ab}},\dots}.
\]
By \vref{lemma:H.plus.action}, change in choice of reduction has no effect on the resulting Dolbeault class.
\begin{lemma}
The maps defined above between cochains: \Slovak{}, Dolbeault, Czechoslovak, \Cech--Dolbeault, are commuting morphisms of differential complexes.
\end{lemma}

\subsection{Wedge and cup product}
We define the cup product of Czechoslovak cochains \(\xi,\eta\) by
\[
(\xi\cup\eta)_{a_0\dots a_k}
\defeq(-1)^{(p+q+r-j)(k-j)}\xi_{a_0\dots a_j}\wedge\eta_{a_j\dots a_k}
\]
if \(\xi\in\gForms[V]{p}[q]{r}\) and \(\eta\in\gForms[W]{p'}[q']{r'}\), giving a cochain valued in the tensor product.
We can check the Leibnitz identity against the differential.

\subsection{Relative cohomology}
Take a subcollection \(I'\) of indices \(a\) from our collection \(I\) of indices \(a\) of open sets \(\subset\set{M_a}\) from our cover.
The open set
\[
M'\defeq \bigcup_{a\in I'} M_a
\]
has an associated collection of Czechoslovak chains built using the sets \(M_a\) with \(a\in I'\).
A \emph{relative cochain} relative to \(I'\) is a cochain \(\xi\) for which every \(\xi_A\) vanishes if all entries of \(A\) come from \(I'\).
Define relative Czechoslovak cohomology as the cohomology of these.
The obvious short exact sequence from relative to all cochains, and then to \(I'\)-cochains gives a long exact sequence in cohomology.

If the open covers are chosen Stein, then the relative \Slovak{} cohomology (in the sense of a flabby resolution of \Slovak{} cohomology) is isomorphic to the relative Czechoslovak cohomology \cite{Suwa2017} p. 3 Theorem 2.4.

Consider the special case when we have a holomorphic Cartan geometry defined on \(M-S\), for some compact complex subvariety \(S\subset M\).
Let \(M_0\defeq M-S\).
The \emph{\Slovak{} cohomology} of \(S\) is
\[
H^{p,q,r}[S,\vb{V}]\defeq \varinjlim_{M_1\supset S} \cohomology{p,q,r}{\cpx{M},\vb{V}},
\]
where the limit is over open sets \(M_1\subseteq M\) containing \(S\).
We hope to compute residue integrals of Chern--Simons invariants here as in \cite{Abate2013}.

\section{Homogeneous examples}
A \emph{complex homogeneous Cartan geometry} is a complex manifold \(X'\) with a holomorphic Cartan geometry, say modelled on \((X,G)\), invariant under the holomorphic action of a complex Lie group \(G'\), i.e. the Lie group acts as holomorphic bundle automorphisms preserving the Cartan connection, acting transitively on \(X'\).
Pick a point \(e_0\in E\) and let \(x_0\in X'\) be its image in \(X'\) and \(H'\defeq {G'}^{x_0}\).
Each \(g\in H'\) determines an element \(h(g)\in H\) by \(ge_0=e_0h(g)\), and holomorphic Lie group morphism \(\map{H'}{H}\).
Let \(G'_0\defeq h^{-1}G_0\).
Then \(\map{G'/G'_0}{E/G'_0}\) is a principal bundle morphism, pulling the sheaf of \Slovak{} cochains back to a sheaf over \(G'/G'_0\).
As a holomorphic principal \(H\)-bundle over \(X'\), \(E_{X'}=\amal{G'}{H'}{H}\).
A real vector subspace of a complex vector space is \emph{totally real} if it contains no complex line; a real submanifold of a complex manifold is \emph{totally real} if its tangent spaces are totally real in the tangent spaces of the complex manifold.
A complex Lie group is \emph{reductive} if it is the complexification of a compact Lie group.
\begin{theorem}\label{thm:homog}
Take a homogeneous holomorphic Cartan geometry \((X',G')\), with notation as above.
Suppose that \(G'\) is reductive, and that the subgroup \(G'_0=h^{-1}G_0\) of \(G'\) is reductive.
Then \Slovak{} cohomology is isomorphic to Dolbeault cohomology via the Dolbeault map.
\end{theorem}
\begin{proof}
The homogeneous space \(G'/G'_0\) is a Stein manifold \cite{Akhiezer1995} p. 160 Corollary 2, \cite{Matsushima/Morimoto:1960} Theorem 1.
The sheaf of \Slovak{} cochains is a sheaf over \(G'/G'_0\), consisting of the holomorphic functions on \(G'\) which are \(G'_0\)-equivariant for the induced \(G'_0\)-representation.
\Slovak{} cohomology is isomorphic to Dolbeault via the Dolbeault isomorphism by \vref{lemma:Stein.H.zero}.
\end{proof}
A \emph{C-space} \((X',G')\) is a simply connected compact complex manifold \(X'\) acted on holomorphically, faithfully and transitively by a connected complex Lie group \(G'\) \cite{Bott:1957,Bott:1988,Griffiths:1962,Griffiths:1963c,Griffiths:1963a,Griffiths:1963b,Wang:1954}.
For example, every generalized complex flag variety is a C-space.
For every C-space \((X',G')\), any Levi factor of \(G'\) acts transitively on \(X'\), so we can replace \(G'\) by that Levi factor, a connected complex semisimple Lie group \cite{Wang:1954} p. 4 theorem 2.2.
Moreover, the maximal compact subgroup of that complex semisimple Lie group is a compact and connected semisimple Lie group also acting transitively on \(X'\) \cite{Wang:1954} p. 13, Theorem 1.
So in particular, \(G'\) can be assumed reductive without loss of generality.
\begin{corollary}
Take a C-space \((X',G')\).
Take a \(G'\)-homogeneous holomorphic Cartan geometry on \(X'\), with \(G'_0=h^{-1}G_0\) reductive as above.
Then \Slovak{} cohomology of that holomorphic Cartan geometry is isomorphic to Dolbeault cohomology via the Dolbeault map.
\end{corollary}
\begin{proof}
We can replace \(G'\) by a Levi factor, so assume that \(G'\) is reductive.
\end{proof}

\subsection{Hirzebruch proportionality}
\begin{corollary}[Hirzebruch proportionality \cite{Hirzebruch:1958}]
Suppose that \(M\) is a compact complex manifold bearing a holomorphic parabolic geometry, i.e. a holomorphic Cartan geometry with model \((X,G)\) is a generalized flag variety.
Associate to each complex polynomial \(p\) in Chern classes in Dolbeault cohomology its integral \(\int_M p(c_1,\dots,c_n)\),
a map
\[
\mapto[\int_M]{p(x_1,\dots,x_n)\in\C{}[x_1,\dots,x_n]}{\int_M p(c_1,\dots,c_n)\in\C{}}.
\]
Then there is a constant \(a_0\in\mathbb{C}\) so that \(\int_M=a_0\int_X\).
\end{corollary}
\begin{remark}
We can't make use of Chern--Simons classes here, since they are \((p+1,p)\) in Dolbeault cohomology, and we don't have a means of computing equations on the usual \deRham{} Chern--Simons classes.
\end{remark}
\begin{proof}
Note that we define these integrals by taking any representative differential form, and they are defined because that form is defined up to adding something from \(\bar\partial\Omega^{n,n-1}=d\Omega^{n,n-1}\subseteq d\Omega^{2n-1}\).
There is a \(G\)-invariant positive line bundle on \(X\), hence we can pick some polynomial \(p_0\), with
\[
\int_X p_0(c_1,\dots,c_n)\ne 0
\]
a nonzero integer.
Every other polynomial \(p\) has
\[
\int_X p(c_1,\dots,c_n) = b_p \int_X p_0(c_1,\dots,c_n)\ne 0
\]
for some complex number \(b_p\), and hence
\[
p(c_1,\dots,c_n) - b_p p_0(c_1,\dots,c_n) = 0
\]
in de~Rham cohomology, so in Dolbeault cohomology by the Hodge isomorphism.
Hence \(\int_M(p-b_pp_0)=0\), so \(\int_M p =b_p\int_M p_0\), and thus
\[
\int_M p=a_0\int_X
\]
where
\[
a_0\coloneq \frac{\int_M p_0}{\int_X p_0}.
\]
\end{proof}
See \cite{Li/Zhao:2017} for calculations of some Chern numbers of flag varieties.

\section{\texorpdfstring{Cartan's $5$ variable paper}{Cartan's 5 variable paper}}
Recall that a \emph{plane field} is a rank \(2\) holomorphic subbundle of the tangent bundle of a complex manifold.
A plane field on a \(5\)-fold \(M\) is \emph{skew} if, near each point of \(M\), it has local holomorphic sections \(X,Y\) so that
\[
X,Y,\lb{X}{Y},\lb{X}{\lb{X}{Y}},\lb{Y}{\lb{X}{Y}}
\]
are linearly independent.
To be more concrete, consider an open subset of \(\C[5]\) with holomorphic coordinates written as
\[
x,y,y',y'',z
\]
with the holomorphic exterior differential system \(\mathscr{I}\) generated by
\[
dy-y'dx,dy'-y''dx,dz-z'(x,y,y',y'',z)dx,
\]
where \(z'=z'(x,y,y',y'')\) is a holomorphic function which is \emph{skew} at the generic point, i.e. satisfies
\[
\pderiv[2]{z'}{(y'')}\ne0.
\]
At points where \(z'\) is skew, the differential forms in \(\mathscr{I}\) vanish precisely on the span of the vectors
\[
X=\partial_{y''},
Y=D_x\defeq\partial_x+y'\partial_y+y''\partial_{y'}+z'\partial_z
\]
which span a skew plane field, as
\[
X,Y,\lb{X}{Y},\lb{X}{\lb{X}{Y}},\lb{Y}{\lb{X}{Y}}
\]
are linearly independent.
Conversely, all skew plane fields on \(5\)-folds occur locally this way \cite{Goursat:1922} section 76 p. 320 (``\emph{En r\'esum\'e, \ldots}''), \cite{Willse:2019} p. 2.
The \emph{Hilbert--Cartan equation} is
\[
z'=\frac{(y'')^2}{2}
\]
for functions \(y=y(x),z=z(x)\).
Cartan \cite{Cartan:30} proved that the skew plane field associated to the Hilbert--Cartan equation has a \(14\)-dimensional Lie algebra of infinitesimal symmetries, acting transitively.
He also proved that it is locally isomorphic to any skew plane field with a Lie algebra of infinitesimal symmetries of dimension more than \(7\).
Let \((X,G_2)\) be the the flag variety with Dynkin diagram \(\dynkin{G}{*x}\), i.e. the unique \(5\)-dimensional complex homogeneous space of \(G_2\) which is not the adjoint variety.
Cartan \cite{Cartan:30} proved that there is a unique \(G_2\)-invariant skew plane field on \(X\).
The skew plane field of the Hilbert--Cartan equation is the pullback of the skew plane field on \(X\), pulled back to the affine open cell of \(X\) (also known as the \emph{open Schubert cell} or the \emph{big cell} \cite{Brion:2005} p. 6), realizing the Lie algebra of symmetry vector fields of the Hilbert--Cartan equations as the Lie algebra of \(G_2\).

In the process of proving these results (among others), Cartan \cite{Cartan:30} associates to each skew plane field a fiber bundle of \nth{2} order frames on which he defines a coframing of \(1\)-forms
\[
ω1,\dots,ω5,ϖ1,\dots,ϖ7,χ1,χ2.
\]
We follow Cartan closely here, and we assume complete familiarity with his paper in this section, so we will provide no further explanation of his paper.
We will see that his bundle, equipped with the \(1\)-forms in his structure equations (suitably corrected, following Robert Bryant), is neither a \(G\)-structure nor a Cartan geometry.
Nonetheless, our theory is robust enough to apply directly to his structure equations, so that without modifying the structure equations, we can compute relations on Chern classes and Chern--Simons invariants in Dolbeault cohomology.

Consider Cartan's equations (5) from \cite{Cartan:30} section 31 p. 149 (p. 967 in \cite{Cartan:II}):
\begin{align*}
d\om1&=\om1\w(2\vp1+\vp4)+\om2\w\vp2+\om3\w\om4,\\
d\om2&=\om1\w\vp3+\om2\w(\vp1+2\vp4)+\om3\w\om5,\\
d\om3&=\om1\w\vp5+\om2\w\vp6+\om3\w(\vp1+\vp4)+\om4\w\om5,\\
d\om4&=\om1\w\vp7+\frac{4}{3}\om3\w\vp6+\om4\w\vp1+\om5\w\vp2,\\
d\om5&=\om2\w\vp7-\frac{4}{3}\om3\w\vp5+\om4\w\vp3+\om5\w\vp4,
\end{align*}
and his equations (8) from the same section (see \vref{table:Cartans.equations}).
Cartan does not provide equations for the exterior derivatives of \(χ1,χ2\); we will provide a partial remedy for this below.

Unfortunately, no one has so far succeeded (to my knowledge) in obtaining explicit coordinate expressions for the \(1\)-forms in Cartan's paper.
It is not difficult to obtain explicit coordinate expressions for \(\om1,\dots,\om5\), but they are very long; I was not able to compute the \(1\)-forms \(ϖ1,\dots,ϖ7,χ1,χ2\).
\begin{table}\caption{Cartan's structure equations from \cite{Cartan:30} section 32 p. 151 (p. 969 in \cite{Cartan:II}), with the sign of \(E\) corrected to agree with the rest of Cartan's paper, thanks to Robert Bryant \cite{Bryant:2022.01.11}}\label{table:Cartans.equations}
\begin{align*}
d\vp1
&=
\vp3\w\vp2
+\frac{1}{3}\om3\w\vp7
-\frac{2}{3}\om4\w\vp5
+\frac{1}{3}\om5\w\vp6
+\om1\wχ1
\\
&+2B_2\om1\w\om3
+B_3\om2\w\om3
+2A_2\om1\w\om4
\\&
+2A_3\om1\w\om5
+A_3\om2\w\om4
+A_4\om2\w\om5,
\\
d\vp2
&=
\vp2\w(\vp1-\vp4)
-\om4\w\vp6
+\om1\wχ2
\\
&+B_4\om2\w\om3
+A_4\om2\w\om4
+A_5\om2\w\om5,
\\
d\vp3
&=
\vp3\w(\vp4-\vp1)
-\om5\w\vp5
+\om2\wχ1
\\&-B_1\om1\w\om3
-A_1\om1\w\om4
-A_2\om1\w\om5,
\\
d\vp4
&=
\vp2\w\vp3+\frac{1}{3}\om3\w\vp7+\frac{1}{3}\om4\w\vp5-\frac{2}{3}\om5\w\vp6+\om2\wχ2\\
&-B_2\om1\w\om3-2B_3\om2\w\om3-A_2\om1\w\om4
\\&-A_3\om1\w\om5
-2A_3\om2\w\om4
-2A_4\om2\w\om5,\\
d\vp5
&=
\vp1\w\vp5+\vp3\w\vp6-\om5\w\vp7+\om3\wχ1\\
&+\frac{9}{32}D_1\om1\w\om2
+\frac{9}{8}C_1\om1\w\om3
+\frac{9}{8}C_2\om2\w\om3
\\
&+
A_2\om3\w\om4
+A_3\om3\w\om5
+\frac{3}{4}B_1\om1\w\om4
\\
&+\frac{3}{4}B_2(\om1\w\om5+\om2\w\om4)
+\frac{3}{4}B_3\om2\w\om5,
\\
d\vp6&=\vp2\w\vp5+\vp4\w\vp6+\om4\w\vp7+\om3\wχ2
\\
&
+\frac{9}{32}D_2\om1\w\om2+\frac{9}{8}C_2\om1\w\om3+\frac{9}{8}C_3\om2\w\om3
\\
&-A_3\om3\w\om4-A_4\om3\w\om5+\frac{3}{4}B_2\om1\w\om4\\
&+\frac{3}{4}B_3(\om1\w\om5+\om2\w\om4)+\frac{3}{4}B_4\om2\w\om5,\\
d\vp7&=\frac{4}{3}\vp5\w\vp6
+(\vp1+\vp4)\w\vp7
+\om4\wχ1
+\om5\wχ2\\
&-\frac{9}{64}E\om1\w\om2
-\frac{3}{8}D_1\om1\w\om3
\\
&
-\frac{3}{8}D_2\om2\w\om3
+2A_3\om4\w\om5
-B_2\om3\w\om4
+B_3\om3\w\om5.
\end{align*}
\end{table}

We point out concerns about Cartan's structure equations.
As there are \(\varpi\wedge\omega\) terms in \(d\varpi\), this is not a \(G\)-structure.
To see that it is also not a Cartan connection, we need to differentiate the structure equations using a computer algebra system, to obtain equations for \(dχ1,dχ2\).
Robert Bryant \cite{Bryant:2013.10.23} alters Cartan's choices of \(χ1,χ2\), defining
\[
\begin{pmatrix}
χ1'\\
χ2'
\end{pmatrix}
\defeq
\begin{pmatrix}
χ1\\
χ2
\end{pmatrix}
+
\begin{pmatrix}
-\frac{3}{8}C_1&-\frac{3}{8}C_2&B_2&A_2&A_3\\
-\frac{3}{8}C_2&-\frac{3}{8}C_3&-B_3&-A_3&-A_4
\end{pmatrix}
\om{}.
\]
Robert notes that \(\omega,\varpi,\chi'\) is a Cartan geometry modelled on the homogeneous space \((X,G_2)\).
Cartan does not provide equations for \(dχ1,dχ2\), but with some effort one can compute that
\[
d
\begin{pmatrix}
χ1'\\
χ2'
\end{pmatrix}
=
\begin{pmatrix}
2ϖ1+ϖ4&ϖ3\\
ϖ2&ϖ1+2ϖ4
\end{pmatrix}
\wedge
\begin{pmatrix}
χ1'\\
χ2'
\end{pmatrix}
+
\begin{pmatrix}
ϖ5\wedgeϖ7\\
ϖ6\wedgeϖ7
\end{pmatrix}
\]
modulo \(\omega\wedge\omega\) curvature terms.

Comparing Cartan's notation \cite{Cartan:30} to ours above,
\[
ω-=
\begin{pmatrix}
ω1\\
ω2\\
ω3\\
ω4\\
ω5
\end{pmatrix}, \quad
ω0=
\begin{pmatrix}
2ϖ1+ϖ4&ϖ2\\
ϖ3&ϖ1+2ϖ4
\end{pmatrix},  \quad
ω+=
\begin{pmatrix}
ϖ5\\
ϖ6\\
ϖ7\\
χ1'\\
χ2'
\end{pmatrix}.
\]
Up to constant nonzero rational number factors, these are the dual \(1\)-forms in the Lie algebra of \(G_2\) corresponding to root vectors in the pattern
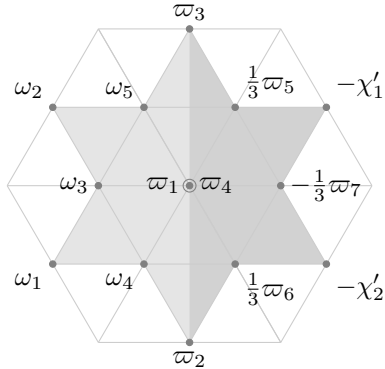
\setlength\weightLength{1.2cm}
\[
\begin{tikzpicture}
\begin{rootSystem}{G}
\roots
\wt[multiplicity=2]{0}{0}
\parabolic{2}
\node[right,black] at (hex cs:x=1,y=0){\(-\frac{1}{3}\vp7\)};
\node[above right,black] at (hex cs:x=1,y=1){\(-χ1'\)};
\node[below right,black] at (hex cs:x=2,y=-1){\(-χ2'\)};
\node[left,black] at (hex cs:x=-1,y=0){\(\om3\)};
\node[above left,black] at (hex cs:x=-2,y=1){\(\om2\)};
\node[below left,black] at (hex cs:x=-1,y=-1){\(\om1\)};
\node[above,black] at (hex cs:x=-1,y=2){\(\vp3\)};
\node[above left,black] at (hex cs:x=-1,y=1){\(\om5\)};
\node[below left,black] at (hex cs:x=0,y=-1){\(\om4\)};
\node[above right,black] at (hex cs:x=0,y=1){\(\frac{1}{3}\vp5\)};
\node[below,black] at (hex cs:x=1,y=-2){\(\vp2\)};
\node[below right,black] at (hex cs:x=1,y=-1){\(\frac{1}{3}\vp6\)};
\node[right,black] at (hex cs:x=0,y=0){\(\vp4\)};
\node[left,black] at (hex cs:x=0,y=0){\(\vp1\)};
%\parabolicgrading
\end{rootSystem}
\end{tikzpicture}
\]
which, with the indicated rescalings as above, I suspect become a Chevalley basis for \(G_2\), although this seems to be difficult to check.
Note that the right half of the root system, together with the points on the vertical through the origin, consists of the root vectors which span the parabolic subgroup of \(G_2\) stabilizing a point of \(X\).
Since \(χ1=χ1'\) and \(χ2=χ2'\) up to holomorphic semibasic torsion terms, we can treat them as equal in calculating \Slovak{} cohomology.
The \Slovak{} differentials are immediate from the equations above for the differentials:
\begin{align*}
\dSω1&=0,\\
\dSω2&=0,\\
\dSω3&=ω1\wedgeϖ5+ω2\wedgeϖ6,\\
\dSω4&=ω1\wedgeϖ7+\frac{4}{3}ω3\wedgeϖ6,\\
\dSω5&=ω2\wedgeϖ7-\frac{4}{3}ω3\wedgeϖ5,\\
\end{align*}
and
\begin{align*}
\dS\vp1
&=
\frac{1}{3}\om3\w\vp7
-\frac{2}{3}\om4\w\vp5
+\frac{1}{3}\om5\w\vp6
+\om1\wχ1,
\\
\dS\vp2
&=
-\om4\w\vp6
+\om1\wχ2,
\\
\dS\vp3
&=
-\om5\w\vp5
+\om2\wχ1,
\\
\dS\vp4
&=
\frac{1}{3}\om3\w\vp7
+\frac{1}{3}\om4\w\vp5
-\frac{2}{3}\om5\w\vp6
+\om2\wχ2
\end{align*}
and
\begin{align*}
\dSϖ5&=0,\\
\dSϖ6&=0,\\
\dSϖ7&=\frac{4}{3}ϖ5\wedgeϖ6,\\
\dS\chi'_1&=ϖ5\wedgeϖ7,\\
\dS\chi'_2&=ϖ6\wedgeϖ7,\\
\end{align*}
For example, the Cartan \(3\)-form is
\begin{align*}
&\phantom{+}ω4ʌϖ1ʌϖ5 + ω5ʌϖ2ʌϖ5 + ω4ʌϖ3ʌϖ6 \\
&+ ω5ʌϖ4ʌϖ6 - 2ω4ʌω5ʌϖ7 - ω3ʌϖ1ʌϖ7 \\
&- ω3ʌϖ4ʌϖ7 - 2ω3ʌω4ʌχ1 - 2ω1ʌϖ1ʌχ1 \\
&- ω2ʌϖ2ʌχ1 - ω1ʌϖ4ʌχ1 - 2ω3ʌω5ʌχ2 \\
&- ω2ʌϖ1ʌχ2 - ω1ʌϖ3ʌχ2 - 2ω2ʌϖ4ʌχ2.
\end{align*}
This is closed in the \Slovak{} differential, giving an element in \Slovak{} cohomology \(H^{1,1,1}\).
The Atiyah class of the plane field is
\begin{align*}
a&=\dSω0,
\\
&=dω0+\frac{1}{2}\lb{ω0}{ω0},
\\
&=
\begin{pmatrix}
2\dS\vp1+\dS\vp4&\dS\vp2\\
\dS\vp3&\dS\vp1+2\dS\vp4
\end{pmatrix},
\\
&=
\begin{pmatrix}
\om3\w\vp7
-\om4\w\vp5
+2\om1\wχ1
+\om2\wχ2
&
-\om4\w\vp6
+\om1\wχ2
\\
-\om5\w\vp5
+\om2\wχ1
&
\om3\w\vp7
-\om5\w\vp6
+\om1\wχ1
+2\om2\wχ2
\end{pmatrix}
\end{align*}

For parabolic geometries, the tangent bundle is not an \(G_0\)-module, if we take \(G_0\) to be a Levi factor for the parabolic subgroup.
But the associated graded of the tangent bundle is an \(G_0\)-module.
The first order \(G\)-structure pseudoconnection \(1\)-form is
\[
\begin{pmatrix}
2\vp1+\vp4&\vp2&0&0&0\\
\vp3&\vp1+2\vp4&0&0&0\\
\vp5&\vp6&\vp1+\vp4&0&0\\
\vp7&0&\frac{4}{3}\vp6&\vp1&\vp2\\
0&\vp7&-\frac{4}{3}\vp5&\vp3&\vp4
\end{pmatrix}.
\]
But the associated graded has representation
\[
\begin{pmatrix}
2\vp1+\vp4&\vp2&0&0&0\\
\vp3&\vp1+2\vp4&0&0&0\\
0&0&\vp1+\vp4&0&0\\
0&0&0&\vp1&\vp2\\
0&0&0&\vp3&\vp4
\end{pmatrix},
\]
so that it is an associated vector bundle of the first order \(G\)-structure.
The reader can then check (by hand!) that
\begin{align*}
c_2&=11(c_1/5)^2,\\
c_3&=13(c_1/5)^3,\\
c_4&=9(c_1/5)^4,\\
c_5&=3(c_1/5)^5
\end{align*}
which agrees with \cite{McKay:2011} (a calculation done by computer).

Consider the Chern--Simons form in \Slovak{} cohomology of the last relation, as shown in \vref{table:CS}.
\begin{table}\caption{A Chern--Simons invariant in \Slovak{} cohomology \(H^{4,1,4}\), arising in the study of skew plane fields on $5$-folds}\label{table:CS}
\begin{align*}
T_{5^5c_5(T)-3c_1(T)^5}&=
5^5T_{c_5(T)}-3T_{c_1(T)^5},\\
&=
\left(\frac{5i}{2\pi}\right)^5\frac{1}{3^3}
\\ &\qquad(
2ω1ʌω3ʌω4ʌω5ʌϖ1ʌϖ5ʌϖ6ʌϖ7ʌχ1
\\ &\qquad
+ ω2ʌω3ʌω4ʌω5ʌϖ2ʌϖ5ʌϖ6ʌϖ7ʌχ1
\\ &\qquad
+ ω1ʌω3ʌω4ʌω5ʌϖ4ʌϖ5ʌϖ6ʌϖ7ʌχ1
\\ &\qquad
+ ω2ʌω3ʌω4ʌω5ʌϖ1ʌϖ5ʌϖ6ʌϖ7ʌχ2
\\ &\qquad
+ ω1ʌω3ʌω4ʌω5ʌϖ3ʌϖ5ʌϖ6ʌϖ7ʌχ2
\\ &\qquad
+ 2ω2ʌω3ʌω4ʌω5ʌϖ4ʌϖ5ʌϖ6ʌϖ7ʌχ2
\\ &\qquad
- 15ω1ʌω2ʌω3ʌω4ʌϖ1ʌϖ5ʌϖ7ʌχ1ʌχ2
\\ &\qquad
- 3ω1ʌω2ʌω3ʌω5ʌϖ2ʌϖ5ʌϖ7ʌχ1ʌχ2
\\ &\qquad
- 12ω1ʌω2ʌω3ʌω4ʌϖ4ʌϖ5ʌϖ7ʌχ1ʌχ2
\\ &\qquad
- 12ω1ʌω2ʌω3ʌω5ʌϖ1ʌϖ6ʌϖ7ʌχ1ʌχ2
\\ &\qquad
- 3ω1ʌω2ʌω3ʌω4ʌϖ3ʌϖ6ʌϖ7ʌχ1ʌχ2
\\ &\qquad
- 15ω1ʌω2ʌω3ʌω5ʌϖ4ʌϖ6ʌϖ7ʌχ1ʌχ2
\end{align*}
\end{table}
For the differential form \(\Psi\) in \vref{table:primitive},
\[
\dS\Psi=T_{5^5c_5(T)-3c_1(T)^5}
\]
in the \Slovak{} differential.
Hence on any reduction of structure group to an \(G_0\)-bundle, the Chern--Simons class of \(\CS{5^5c_5(T)-3c_1(T)^5}\) vanishes in Dolbeault cohomology.
So any \(5\)-fold with a skew plane field has not only the above relations among Chern classes in Dolbeault cohomology, but also the vanishing of this Chern--Simons invariant in Dolbeault cohomology.
Moreover, there is no residue from that Chern--Simons form on the singular hypersurface of a meromorphic skew plane field on a complex \(5\)-fold.
\begin{table}\caption{A primitive in \Slovak{} cohomology \(H^{2,3,2}+H^{3,2,3}\) for the Chern--Simons invariant of \vref{table:CS}}\label{table:primitive}
\begin{align*}\Psi
&=
\left(\frac{5i}{2\pi}\right)^5\frac{2^2}{3}
\\ &\quad
(
7ω2ʌω4ʌϖ1ʌϖ2ʌϖ3ʌϖ5ʌϖ6ʌϖ7
\\ &\quad
+ 2ω1ʌω5ʌϖ1ʌϖ2ʌϖ3ʌϖ5ʌϖ6ʌϖ7
\\ &\quad
- 5ω2ʌω5ʌϖ1ʌϖ2ʌϖ4ʌϖ5ʌϖ6ʌϖ7
\\ &\quad
- 5ω1ʌω4ʌϖ1ʌϖ3ʌϖ4ʌϖ5ʌϖ6ʌϖ7
\\ &\quad
+ 2ω2ʌω4ʌϖ2ʌϖ3ʌϖ4ʌϖ5ʌϖ6ʌϖ7
\\ &\quad
+ 7ω1ʌω5ʌϖ2ʌϖ3ʌϖ4ʌϖ5ʌϖ6ʌϖ7
\\ &\quad
+ ω2ʌω4ʌω5ʌϖ1ʌϖ2ʌϖ5ʌϖ6ʌχ1
\\ &\quad
- 6ω1ʌω4ʌω5ʌϖ2ʌϖ3ʌϖ5ʌϖ6ʌχ1
\\ &\quad
- ω1ʌω4ʌω5ʌϖ1ʌϖ4ʌϖ5ʌϖ6ʌχ1
\\ &\quad
+ 2ω2ʌω4ʌω5ʌϖ2ʌϖ4ʌϖ5ʌϖ6ʌχ1
\\ &\quad
+ 5ω2ʌω3ʌω4ʌϖ1ʌϖ2ʌϖ5ʌϖ7ʌχ1
\\ &\quad
+ ω1ʌω3ʌω5ʌϖ1ʌϖ2ʌϖ5ʌϖ7ʌχ1
\\ &\quad
- 4ω2ʌω3ʌω4ʌϖ2ʌϖ4ʌϖ5ʌϖ7ʌχ1
\\ &\quad
- ω1ʌω3ʌω5ʌϖ2ʌϖ4ʌϖ5ʌϖ7ʌχ1
\\ &\quad
+ 6ω1ʌω2ʌϖ1ʌϖ2ʌϖ4ʌϖ5ʌϖ7ʌχ1
\\ &\quad
+ 4ω2ʌω3ʌω5ʌϖ1ʌϖ2ʌϖ6ʌϖ7ʌχ1
\\ &\quad
- ω2ʌω3ʌω4ʌϖ2ʌϖ3ʌϖ6ʌϖ7ʌχ1
\\ &\quad
- 5ω2ʌω3ʌω5ʌϖ2ʌϖ4ʌϖ6ʌϖ7ʌχ1
\\ &\quad
- 2ω1ʌω4ʌω5ʌϖ1ʌϖ3ʌϖ5ʌϖ6ʌχ2
\\ &\quad
+ 6ω2ʌω4ʌω5ʌϖ2ʌϖ3ʌϖ5ʌϖ6ʌχ2
\\ &\quad
+ ω2ʌω4ʌω5ʌϖ1ʌϖ4ʌϖ5ʌϖ6ʌχ2
\\ &\quad
- ω1ʌω4ʌω5ʌϖ3ʌϖ4ʌϖ5ʌϖ6ʌχ2
\\ &\quad
+ 5ω1ʌω3ʌω4ʌϖ1ʌϖ3ʌϖ5ʌϖ7ʌχ2
\\ &\quad
+ ω1ʌω3ʌω5ʌϖ2ʌϖ3ʌϖ5ʌϖ7ʌχ2
\\ &\quad
- 4ω1ʌω3ʌω4ʌϖ3ʌϖ4ʌϖ5ʌϖ7ʌχ2
\\ &\quad
+ ω2ʌω3ʌω4ʌϖ1ʌϖ3ʌϖ6ʌϖ7ʌχ2
\\ &\quad
+ 4ω1ʌω3ʌω5ʌϖ1ʌϖ3ʌϖ6ʌϖ7ʌχ2
\\ &\quad
- ω2ʌω3ʌω4ʌϖ3ʌϖ4ʌϖ6ʌϖ7ʌχ2
\\ &\quad
- 5ω1ʌω3ʌω5ʌϖ3ʌϖ4ʌϖ6ʌϖ7ʌχ2
\\ &\quad
+ 6ω1ʌω2ʌϖ1ʌϖ3ʌϖ4ʌϖ6ʌϖ7ʌχ2
)
\end{align*}
\end{table}

\section{Conclusion}
We can now employ constraints on characteristic classes and Chern--Simons classes from generalized flag varieties to yield constraints on Dolbeault cohomology of complex manifolds carrying holomorphic geometric structures with a wide selection of different models.
It remains to make a theory of such structures on singular varieties, generalizing the theory of singular locally Hermitian symmetric varieties, on which some results about characteristic class invariants are known \cite{Mumford:1977} which generalize Hirzebruch's proportionality theorem.
The results above should help in classifying all holomorphic Cartan geometries which are both modelled on and also defined on C-spaces.
There are many interesting examples yet to compute of relations on Chern--Simons invariants following our method above.

\nocite{sagemath}
\bibliographystyle{amsplain}
\bibliography{chern-cartan-2}

\def\cprime{$'$} \def\BallicoTitle{${\mathbf C}^2 \backslash \{0\}$}
  \def\GNtitle{${\rm GL}(2,\Bbb R)$} \def\rfour{${\bf R}\sp 4$}
\providecommand{\bysame}{\leavevmode\hbox to3em{\hrulefill}\thinspace}
\providecommand{\MR}{\relax\ifhmode\unskip\space\fi MR }
\providecommand{\MRhref}[2]{%
  \href{http://www.ams.org/mathscinet-getitem?mr=#1}{#2}
}
\providecommand{\href}[2]{#2}
\begin{thebibliography}{10}

\bibitem{Abate2013}
Marco Abate, Filippo Bracci, Tatsuo Suwa, and Francesca Tovena,
  \emph{Localization of {A}tiyah classes}, Rev. Mat. Iberoam. \textbf{29}
  (2013), no.~2, 547--578. \MR{3047428}

\bibitem{Akhiezer1995}
Dmitri~N. Akhiezer, \emph{Lie group actions in complex analysis}, Aspects of
  Mathematics, E27, Friedr. Vieweg \& Sohn, Braunschweig, 1995. \MR{1334091}

\bibitem{Atiyah1957}
M.~F. Atiyah, \emph{Complex analytic connections in fibre bundles}, Trans.
  Amer. Math. Soc. \textbf{85} (1957), 181--207. \MR{MR0086359 (19,172c)}

\bibitem{Baum/Bott:1972}
Paul Baum and Raoul Bott, \emph{Singularities of holomorphic foliations}, J.
  Differential Geometry \textbf{7} (1972), 279--342. \MR{0377923 (51 \#14092)}

\bibitem{Biswas:2017}
Indranil Biswas, \emph{On connections on principal bundles}, Arab J. Math. Sci.
  \textbf{23} (2017), no.~1, 32--43. \MR{3589498}

\bibitem{Bott:1957}
Raoul Bott, \emph{Homogeneous vector bundles}, Ann. of Math. (2) \textbf{66}
  (1957), 203--248. \MR{MR0089473 (19,681d)}

\bibitem{Bott:1988}
\bysame, \emph{On induced representations}, The mathematical heritage of
  {H}ermann {W}eyl ({D}urham, {NC}, 1987), Proc. Sympos. Pure Math., vol.~48,
  Amer. Math. Soc., Providence, RI, 1988, pp.~1--13. \MR{974328}

\bibitem{Bott/Tu:1982}
Raoul Bott and Loring~W. Tu, \emph{Differential forms in algebraic topology},
  Graduate Texts in Mathematics, vol.~82, Springer-Verlag, New York-Berlin,
  1982. \MR{658304}

\bibitem{Brion:2005}
Michel Brion, \emph{Lectures on the geometry of flag varieties}, Topics in
  cohomological studies of algebraic varieties, Trends Math., Birkh\"{a}user,
  Basel, 2005, pp.~33--85. \MR{2143072}

\bibitem{Bryant:2013.10.23}
Robert Bryant, personal communication, October 2013.

\bibitem{Bryant:2022.01.11}
\bysame, personal communication, January 2022.

\bibitem{Buchdahl:1983}
N.~Buchdahl, \emph{On the relative de {R}ham sequence}, Proc. Amer. Math. Soc.
  \textbf{87} (1983), no.~2, 363--366. \MR{681850}

\bibitem{Cap.Slovak.2009a}
Andreas {\v{C}}ap and Jan Slov{\'a}k, \emph{Parabolic geometries. {I}},
  Mathematical Surveys and Monographs, vol. 154, American Mathematical Society,
  Providence, RI, 2009, Background and general theory. \MR{2532439
  (2010j:53037)}

\bibitem{Cartan:30}
\'Elie Cartan, \emph{Les syst\`emes de {P}faff \`a cinq variables et les
  \'equations aux d\'eriv\'ees partielle du second ordre}, Ann. {\'E}c. Norm.
  \textbf{27} (1910), 109--192, Also in \cite{Cartan:II}, pp. 927--1010.

\bibitem{Cartan:II}
\bysame, \emph{{\OE}uvres compl\`etes. {P}artie {II}}, second ed., \'Editions
  du Centre National de la Recherche Scientifique (CNRS), Paris, 1984,
  Alg\`ebre, syst\`emes diff\'erentiels et probl\`emes d'\'equivalence.
  [Algebra, differential systems and problems of equivalence]. \MR{85g:01032b}

\bibitem{Cartan:1992}
\bysame, \emph{Le\c cons sur la g\'eom\'etrie projective complexe. {L}a
  th\'eorie des groupes finis et continus et la g\'eom\'etrie diff\'erentielle
  trait\'ees par la m\'ethode du rep\`ere mobile. {L}e\c cons sur la th\'eorie
  des espaces \`a connexion projective}, Les Grands Classiques
  Gauthier-Villars. [Gauthier-Villars Great Classics], \'Editions Jacques
  Gabay, Sceaux, 1992, Reprint of the editions of 1931, 1937 and 1937.
  \MR{1190006 (93i:01030)}

\bibitem{Chern/Simons:1974}
Shiing~Shen Chern and James Simons, \emph{Characteristic forms and geometric
  invariants}, Ann. of Math. (2) \textbf{99} (1974), 48--69. \MR{353327}

\bibitem{Eastwood/Gindikin/Wong:1996}
Michael~G. Eastwood, Simon~G. Gindikin, and Hon-Wai Wong, \emph{A holomorphic
  realization of analytic cohomology}, C. R. Acad. Sci. Paris S\'er. I Math.
  \textbf{322} (1996), no.~6, 529--534. \MR{1383430}

\bibitem{Gardner:1989}
Robert~B. Gardner, \emph{The method of equivalence and its applications},
  CBMS-NSF Regional Conference Series in Applied Mathematics, vol.~58, Society
  for Industrial and Applied Mathematics (SIAM), Philadelphia, PA, 1989.
  \MR{MR1062197 (91j:58007)}

\bibitem{Goursat:1922}
\'Edouard Goursat, \emph{Le\c{c}ons sur le probl\`eme de {P}faff}, Librairie
  Scientifique, J. Hermann, Paris, 1922.

\bibitem{Griffiths:1962}
Phillip~A. Griffiths, \emph{On certain homogeneous complex manifolds}, Proc.
  Nat. Acad. Sci. U.S.A. \textbf{48} (1962), 780--783. \MR{0137128}

\bibitem{Griffiths:1963c}
\bysame, \emph{On the differential geometry of homogeneous vector bundles},
  Trans. Amer. Math. Soc. \textbf{109} (1963), 1--34. \MR{0162248}

\bibitem{Griffiths:1963a}
\bysame, \emph{Some geometric and analytic properties of homogeneous complex
  manifolds. {I}. {S}heaves and cohomology}, Acta Math. \textbf{110} (1963),
  115--155. \MR{0149506}

\bibitem{Griffiths:1963b}
\bysame, \emph{Some geometric and analytic properties of homogeneous complex
  manifolds. {II}. {D}eformation and bundle theory}, Acta Math. \textbf{110}
  (1963), 157--208. \MR{0154300}

\bibitem{Hilgert.Neeb:2012}
Joachim Hilgert and Karl-Hermann Neeb, \emph{Structure and geometry of {L}ie
  groups}, Springer Monographs in Mathematics, Springer, New York, 2012.
  \MR{3025417}

\bibitem{Hirzebruch:1958}
Friedrich Hirzebruch, \emph{Automorphe {F}ormen und der {S}atz von
  {R}iemann-{R}och}, Symposium internacional de topolog\'\i a algebraica
  {I}nternational symposium on algebraic topology, Universidad Nacional
  Aut\'onoma de M\'exico and UNESCO, Mexico City, 1958, pp.~129--144.
  \MR{0103280}

\bibitem{Hochschild:1981}
Gerhard~P. Hochschild, \emph{Basic theory of algebraic groups and {L}ie
  algebras}, Graduate Texts in Mathematics, vol.~75, Springer-Verlag, New
  York-Berlin, 1981. \MR{620024}

\bibitem{Jahnke/Radloff:2018}
Priska Jahnke and Ivo Radloff, \emph{Projective manifolds modeled after
  hyperquadrics}, Internat. J. Math. \textbf{29} (2018), no.~14, 1850100, 20.
  \MR{3900874}

\bibitem{Kamber/Tondeur:1975}
Franz~W. Kamber and Philippe Tondeur, \emph{Foliated bundles and characteristic
  classes}, Lecture Notes in Mathematics, Vol. 493, Springer-Verlag, Berlin-New
  York, 1975. \MR{0402773}

\bibitem{Knapp:2002}
Anthony~W. Knapp, \emph{Lie groups beyond an introduction}, second ed.,
  Progress in Mathematics, vol. 140, Birkh\"auser Boston Inc., Boston, MA,
  2002. \MR{MR1920389 (2003c:22001)}

\bibitem{KobayashiOchiai:1980}
Sh{\^o}shichi Kobayashi and Takushiro Ochiai, \emph{Holomorphic projective
  structures on compact complex surfaces}, Math. Ann. \textbf{249} (1980),
  no.~1, 75--94. \MR{81g:32021}

\bibitem{Kobayashi/Ochiai:1982}
\bysame, \emph{Holomorphic structures modeled after hyperquadrics}, Tohoku
  Math. J. (2) \textbf{34} (1982), no.~4, 587--629. \MR{685426}

\bibitem{Li/Zhao:2017}
Ping Li and Wenjing Zhao, \emph{Combinatorial identities and {C}hern numbers of
  complex flag manifolds}, math.DG/1702.01698, 2017.

\bibitem{Matsushima/Morimoto:1960}
Yoz{\^o} Matsushima and Akihiko Morimoto, \emph{Sur certains espaces fibr\'es
  holomorphes sur une vari\'et\'e de {S}tein}, Bull. Soc. Math. France
  \textbf{88} (1960), 137--155. \MR{MR0123739 (23 \#A1061)}

\bibitem{McKay:2011}
Benjamin McKay, \emph{Characteristic forms of complex {C}artan geometries},
  Adv. Geom. \textbf{11} (2011), no.~1, 139--168. \MR{2770434}

\bibitem{mckay2023introduction}
Benjamin McKay, \emph{An introduction to cartan geometries}, 2023.

\bibitem{Molzon/Mortensen:1996}
Robert Molzon and Karen~Pinney Mortensen, \emph{The {S}chwarzian derivative for
  maps between manifolds with complex projective connections}, Trans. Amer.
  Math. Soc. \textbf{348} (1996), no.~8, 3015--3036. \MR{MR1348154 (96j:32028)}

\bibitem{Mumford:1977}
David Mumford, \emph{Hirzebruch's proportionality theorem in the noncompact
  case}, Invent. Math. \textbf{42} (1977), 239--272. \MR{471627}

\bibitem{Rovenskii:1998}
Vladimir~Y. Rovenskii, \emph{Foliations on {R}iemannian manifolds and
  submanifolds}, Birkh\"auser Boston Inc., Boston, MA, 1998, Appendix A by the
  author and Victor Toponogov. \MR{1486826 (99b:53043)}

\bibitem{Suwa:2000}
Tatsuo Suwa, \emph{Dual class of a subvariety}, Tokyo J. Math. \textbf{23}
  (2000), no.~1, 51--68. \MR{1763504}

\bibitem{Suwa2017}
\bysame, \emph{Relative cohomology for the sections of a complex of fine
  sheaves}, Kyoto University Research Information Repository (2017), 1--17.

\bibitem{sagemath}
{The Sage Developers}, \emph{{S}agemath, the {S}age {M}athematics {S}oftware
  {S}ystem ({V}ersion 9.2.0)}, 2021, {\tt https://www.sagemath.org}.

\bibitem{Varadarajan:1984}
V.~S. Varadarajan, \emph{Lie groups, {L}ie algebras, and their
  representations}, Graduate Texts in Mathematics, vol. 102, Springer-Verlag,
  New York, 1984, Reprint of the 1974 edition. \MR{746308}

\bibitem{Wang:1954}
Hsien-Chung Wang, \emph{Closed manifolds with homogeneous complex structure},
  Amer. J. Math. \textbf{76} (1954), 1--32. \MR{MR0066011 (16,518a)}

\bibitem{Willse:2019}
Travis Willse, \emph{Homogeneous real {$(2,3,5)$} distributions with isotropy},
  SIGMA Symmetry Integrability Geom. Methods Appl. \textbf{15} (2019), Paper
  No. 008, 28. \MR{3907747}

\end{thebibliography}
\end{document}